\documentclass[10 pt, leqno]{article}
\date{}
\usepackage{amssymb, amsbsy, amsmath, amsfonts, amssymb, amscd, mathrsfs, amsthm, dsfont}

\usepackage[english]{babel}
\usepackage[T1]{fontenc}
\usepackage{indentfirst}
\usepackage{color}
\usepackage[all]{xy}
 
\makeatletter
\@addtoreset{equation}{section}
\makeatother

\newtheorem{statement}{}[section]
\newtheorem{theorem}[statement]{Theorem}
\newtheorem{lemma}[statement]{Lemma}

\newtheorem{proposition}[statement]{Proposition}
\newtheorem{definition}[statement]{Definition}
\newtheorem{corollary}[statement]{Corollary}

\newcommand\C{\mathbb C}

\newcommand\R{\mathbb R}
\newcommand\T{\mathbb T}
\newcommand\D{\mathbb D}

\newcommand\e{{\rm e}}

\newcommand\eps{\varepsilon}
\newcommand\ind{\mathds{1}}

\renewcommand \Re{{\mathfrak R}{\rm e}\,}
\renewcommand \Im{{\mathfrak I}{\rm m}\,}
\newcommand\converge{\mathop{\longrightarrow}\limits}

\let\phi=\varphi

\let\tilde=\widetilde
\newcommand\tq{\, ; \ }

\title{\bf Comparison of singular numbers of composition operators on different Hilbert spaces of analytic functions}
\author{\it Pascal~Lef\`evre, Daniel~Li,  \\ \it Herv\'e~Queff\'elec, Luis~Rodr{\'\i}guez-Piazza}

\date{\footnotesize \today}

\begin{document}

\maketitle

\noindent {\bf Abstract.} We compare the rate of decay of singular numbers of a given composition operator acting on various Hilbert spaces of analytic 
functions on the unit disk $\D$.  We show that for the Hardy and Bergman spaces, our results are sharp. 
We also give lower and upper estimates of the singular numbers of the composition operator with symbol the ``cusp map'' and the lens maps, acting on weighted 
Dirichlet spaces. 
\medskip

\noindent {\bf MSC 2010} primary: 47B33 ; secondary: 46B28 ; 30H10 ; 30H20
\smallskip

\noindent {\bf Key-words} approximation numbers ; composition operator ; Hardy space ; Hilbert spaces of analytic functions ; Schatten classes ; 
singular numbers ; weighted Bergman space ; weighted Dirichlet space

\section {Introduction} 

Composition operators are mainly studied on Hilbert spaces of analytic functions, and more specifically on the Hardy space $H^2$, the Bergman space  
${\mathfrak B}^2$, and the Dirichlet space ${\mathcal D} = {\mathcal D}^{\, 2}$. It is well known, thanks to the Littlewood subordination principle, that 
every analytic self-map $\phi \colon \D \to \D$ induces a bounded composition operator $C_\phi$ on $H^2$ and on ${\mathfrak B}^2$, but not necessarily 
on ${\mathcal D}^{\, 2}$ (\cite[Chapter~1 and Exercises]{Shapiro-livre}; see also \cite[Section~6.2]{Primer-Dirichlet}). There exist even composition operators 
which are not bounded on ${\mathcal D}^2$ 
but which are in all Schatten classes $S_p (H^2)$ and $S_p ({\mathfrak B}^2)$ with $p > 0$, of both the Hardy space and the Bergman space 
(\cite[Theorem~2.10]{LLQR-JFA2013}). Nevertheless, for compact composition operators, the following results hold: 1) every composition operator which is 
compact on $H^2$ is compact on ${\mathfrak B}^2$ (see \cite[Theorem~3.5]{Shapiro-livre} and \cite[Theorem~3.5]{MacCluer-Shapiro}); 2) every 
composition operator that is compact on ${\mathcal D}^{\, 2}$ is in all the Schatten classes $S_p (H^2)$ for $p > 0$ (\cite[Theorem~2.9]{LLQR-JFA2013}); for 
every $p > 0$, every composition operator that is in $S_p (H^2)$ is in $S_p ({\mathfrak B}^2)$. Since the membership in a Schatten class $S_p$ of an operator on 
a Hilbert space means that its approximation numbers are $\ell_p$-summable, that suggests that there is a strong link between the approximation numbers 
$a_n^{{\mathcal D}^{\, 2}} (C_\phi)$, $a_n^{H^2} (C_\phi)$ and $a_n^{{\mathfrak B}^2} (C_\phi)$ of the composition operator $C_\phi$ on 
${\mathcal D}^{\, 2}$, $H^2$ and ${\mathfrak B}^2$ respectively. 
\smallskip

The aim of this paper is to prove  that, indeed, in some sense $a_n^{{\mathcal D}^{\, 2}} (C_\phi)$ is ``greater'' than $a_n^{H^2} (C_\phi)$, which is ``greater'' 
than $a_n^{{\mathfrak B}^2} (C_\phi)$. We recover then that $C_\phi \in S_p (H^2)$ implies that $C_\phi \in S_p ({\mathfrak B}^2)$ 
(Section~\ref{comparison}). In Section~\ref{sec: conditional multipliers}, we also give some results about conditional multipliers. 

In Section~\ref{Hardy-Bergman} we give an example with $C_\phi$ compact on $H^2$ but not in any Schatten class 
$S_p ({\mathfrak B}^2)$ for $p < \infty$. We prove that $C_\phi \in S_p (H^2)$ implies that $C_\phi \in S_{p/2} ({\mathfrak B}^2)$ and give an example 
with $C_\phi \in S_p (H^2)$ but $C_\phi \notin S_q ({\mathfrak B}^2)$ for any $q < p/2$. 
\smallskip

However, our result is not sufficient to explain why the compactness of $C_\phi$ on ${\mathcal D}^{\, 2}$ implies that $C_\phi \in S_p (H^2)$ for all $p > 0$. 
A more subtle relationship should exist between $a_n^{{\mathcal D}^{\, 2}} (C_\phi)$ and $a_n^{H^2} (C_\phi)$.  In fact, for every composition operator 
$C_\phi$ that is compact on ${\mathcal D}^{\, 2}$, we have 
$\lim_{n \to \infty} \big[ a_n^{{\mathcal D}^{\, 2}} (C_\phi) \big]^{1 / n} = \lim_{n \to \infty} \big[ a_n^{H^2} (C_\phi) \big]^{1 / n}$ 
(\cite[Theorem~3.1 and Theorem~3.14]{LQR-radius}); in particular, for symbols $\phi$ such that $\| \phi \|_\infty < 1$, the numbers  
$a_n^{{\mathcal D}^{\, 2}} (C_\phi)$ and $a_n^{H^2} (C_\phi)$ behave like $r^n$, with $r = \exp ( - 1 / {\rm Cap} [\phi (\D)] )$, and where 
${\rm Cap} [\phi (\D)]$ is the Green capacity of $\phi (\D)$. On the other hand, for the so-called cusp map $\chi$, we have, for some constants $c_1 > c_1' > 0$ 
(\cite[Theorem~4.3]{LQR-Fennicae}):
\begin{equation} \label{n sur log n}
\e^{- {c_1} n / \log n} \lesssim a_n^{H^2} (C_\chi) \lesssim \e^{- {c_1'} n / \log n} 
\end{equation} 
and, for some constants $c_2 > c_2' > 0$ (\cite[Theorem~3.1]{LLQR-Arkiv}):
\begin{equation} \label{cusp-Dirichlet} 
\e^{- {c_2} \sqrt{n}} \lesssim a_n^{{\mathcal D}^{\, 2}} (C_\chi) \lesssim \e^{- {c_2'} \sqrt{n}} \, ,
\end{equation} 
which is much greater. In Section~\ref{cusp}, we show that the behavior of $a_n (C_\chi)$ in \eqref{n sur log n} holds in all weighted Dirichlet spaces 
${\mathcal D}^{\, 2}_\alpha$ for $\alpha > 0$ (with other constants), and hence \eqref{cusp-Dirichlet} shows that a jump happens for $\alpha = 0$. We also
look at the lens maps.

\section {Notation and background}

Let $\D$ be the open unit disk in $\C$. We denote $dA = dx \, dy / \pi$ the normalized area measure on $\D$. The normalized Lebesgue measure 
$dt / 2 \pi$ on $\T = \partial \D$ is denoted $dm$. 

\subsection {Hilbert spaces of analytic functions}

Recall that the Hardy space $H^2$ is the space of analytic functions $f \colon \D \to \C$ such that:
\begin{displaymath} 
\| f \|_{H^2}^2 := \sup_{0 < r < 1} \int_\T |f (r \xi) |^2 \, d m (\xi) < \infty \, .
\end{displaymath} 
If $f (z) = \sum_{k = 0}^\infty c_k z^k$, we have $\| f \|_{H^2}^2 = \sum_{k = 0}^\infty |c_k|^2$. 
\smallskip

The Bergman space ${\mathfrak B}^2$ is the space of analytic functions $f \colon \D \to \C$ such that:
\begin{displaymath} 
\| f \|_{{\mathfrak B}^2}^2 := \int_\D |f (z) |^2 \, dA (z) < \infty \, .
\end{displaymath} 
If $f (z) = \sum_{k = 0}^\infty c_k z^k$, we have $\| f \|_{{\mathfrak B}^2}^2 = \sum_{k = 0}^\infty \frac{|c_k|^2}{k + 1}\,$.
\smallskip\goodbreak

More generally, for $\gamma > - 1$, the \emph{weighted Bergman space} ${\mathfrak B}_\gamma^2$ is the space of analytic functions $f \colon \D \to \C$ 
such that:
\begin{displaymath} 
\| f \|_{{\mathfrak B}_\gamma^2}^2 = (\gamma + 1) \int_\D |f (z)|^2 (1 - |z|^2)^\gamma \, dA (z) < \infty \, ,
\end{displaymath} 
and if $f (z) = \sum_{k = 0}^\infty c_k z^k$, we have:
\begin{displaymath} 
\| f \|_{{\mathfrak B}_\gamma^2}^2 = \sum_{k = 0}^\infty \beta_k |c_k|^2 \, ,
\end{displaymath} 
with:
\begin{displaymath} 
\beta_k = \frac{k! \, \Gamma (\gamma + 2)}{\Gamma (k + \gamma + 2)} \approx \frac{1}{(k + 1)^{\gamma + 1}} 
\end{displaymath} 
(the equivalence depends on $\gamma$).
\smallskip

Hence ${\mathfrak B}^2 = {\mathfrak B}_0^2$ and $H^2$ corresponds to the degenerate case $\gamma = - 1$.
\smallskip

The Dirichlet space ${\mathcal D}^{\, 2}$ is the space of analytic functions  $f \colon \D \to \C$ such that:
\begin{displaymath} 
\| f \|_{{\mathcal D}^{\, 2}}^2 = |f ( 0)|^2 + \int_\D |f ' (z)|^2 \, dA (z) < \infty \, .
\end{displaymath} 
If $f (z) = \sum_{k = 0}^\infty c_k z^k$, we have $\| f \|_{{\mathcal D}^{\, 2}}^2 = |c_0|^2 + \sum_{k = 0}^\infty k \, |c_k|^2$.
\medskip 

With the equivalent norm $\|| f |\|_{{\mathcal D}^{\, 2}}^2 = \| f \|_{H^2}^2 +  \int_\D |f ' (z)|^2 \, dA (z)$, we have the more pleasant 
form $\|| f |\|_{{\mathcal D}^{\, 2}}^2 = \sum_{k = 0}^\infty (k + 1) \, |c_k|^2$.
\smallskip

More generally, for $\alpha > - 1$, the \emph{weighted Dirichlet space} ${\mathcal D}_\alpha^{\, 2}$ is the space of analytic functions  $f \colon \D \to \C$ 
such that:
\begin{displaymath} 
\| f \|_{{\mathcal D}_\alpha^{\, 2}}^2 = |f ( 0)|^2 + (\alpha + 1) \int_\D |f ' (z)|^2 (1 - |z|^2)^\alpha \, dA (z) < \infty \, ,
\end{displaymath} 
and if $f (z) = \sum_{k = 0}^\infty c_k z^k$, we have:
\begin{displaymath} 
\| f \|_{{\mathcal D}_\alpha^{\, 2}}^2 = \sum_{k = 0}^\infty \beta_k |c_k|^2 \, , 
\end{displaymath} 
with $\beta_0 = 1$ and for $k \geq 1$:
\begin{displaymath} 
\beta_k = \frac{k \, .\, k! \, \Gamma (\alpha + 2)}{\Gamma (k + \alpha + 1)} \approx \frac{1}{(k + 1)^{\alpha - 1}} 
\end{displaymath} 
(the equivalence depending on $\alpha$). Another equivalent expression is:
\begin{displaymath} 
\tilde \beta_k = \frac{(k + 1)! \, \Gamma (\alpha + 2)}{\Gamma (k + \alpha + 1)} \, ;
\end{displaymath} 
we have $\beta_k \leq \tilde \beta_k \leq 2 \beta_k$. 
\smallskip

In particular, for $\gamma > - 1$: 
\begin{equation} 
{\mathcal D}^{\, 2} = {\mathcal D}_0^{\, 2} \, , \quad H^2 = {\mathcal D}_1^{\, 2} \quad \text{and} \quad 
{\mathfrak B}^2_\gamma = {\mathcal D}_{\gamma + 2}^{\, 2} \, .
\end{equation} 
%
\subsection {Composition operators} 

Any analytic self-map $\phi \colon \D \to \D$ defines a bounded composition operator $C_\phi$ on the Hardy space $H^2$ (see 
\cite[Section~2.2]{MacCluer-Shapiro}) and on every weighted Bergman space ${\mathfrak B}_\gamma^2$ for $\gamma > - 1$ 
(\cite[Proposition~3.4]{MacCluer-Shapiro}), hence on every weighted Dirichlet space ${\mathcal D}_\alpha^{\, 2}$ with $\alpha \geq 1$. 
However, this is not always the case on the weighted Dirichlet spaces ${\mathcal D}_\alpha^{\, 2}$ for $\alpha < 1$ 
(\cite[Proposition~3.12]{MacCluer-Shapiro}). 

For convenience, we assume that $\phi$ is not constant and we say that $\phi$ is a \emph{symbol}. We denote $\phi^\ast$ the boundary values function of 
$\phi$. 

\medskip
The Carleson window of size $h$ centered at $\xi \in \T$ is:
\begin{equation} 
W (\xi, h) = \{ z \in \D \tq |z| \geq 1 - h \quad \text{and} \quad - \pi h \leq \arg (\bar{\xi} z)  < \pi h \} \, .
\end{equation} 
For every integer $n \geq 1$ and for $j = 0, \ldots, 2^n - 1$, we set:
\begin{equation} 
W_{n, j} = W (\e^{2 j i \pi / 2^n}, 2^{- n}) \, .
\end{equation} 

We also use the Carleson boxes
\begin{equation}
S (\xi, h) = \{z \in \D \tq |\xi - z| < h\} \, ,
\end{equation}
which satisfy $S (\xi, h) \subseteq W (\xi, h) \subseteq S (\xi, 2 \pi h)$.

The Hastings-Luecking boxes are defined, for every integer $n \geq 1$ and for $0 \leq j \leq 2^n -1$, as: 
\begin{equation}
R_{n, j} = \Big\{z\in \D \tq 1 - \frac{1}{2^{n - 1}} \leq |z| < 1 - \frac{1}{2^{n}} \text{ and }
\frac{2 j \pi}{2^n} \leq \arg z < \frac{2 (j + 1) \pi }{2^n}\,\Big\}
\end{equation}

A measure $\mu$ on $\overline \D$ is a \emph{Carleson measure} if $\sup_{\xi \in \T} \mu \big[ \overline{W (\xi, h)} \big] = {\rm O}\, (h)$. By the 
Carleson embedding theorem, $\mu$ is a Carleson measure if and only if the inclusion map $J_\mu \colon H^2 \to L^2 (\mu)$ is bounded. The automatic 
boundedness of $C_\phi$ on $H^2$ implies that the pull-back measure $m_\phi$, defined as $m_\phi (B) = m [{\phi^\ast}^{ - 1} (B)]$ for all Borel sets 
$B \subseteq \overline \D$, is a Carleson measure. This composition operator is compact on $H^2$ if and only if $m_\phi$ is supported by $\D$ and 
$\sup_{\xi \in \T} m_\phi [W (\xi, h)] = {\rm o}\, (h)$ (\cite{MacCluer}). 

Similar results hold for composition operators on the weighted Bergman spaces (\cite[Theorem~4.3]{MacCluer-Shapiro}. 

\subsection {Singular numbers, approximation numbers and Schatten classes} 

Let $H$ be a separable complex Hilbert space and $T \colon H \to H$ be a compact operator. There exist two orthonormal sequences $(u_n)$ and $(v_n)$ and 
a non-increasing sequence $(s_n)$ of non-negative numbers with $s_n \converge_{n \to \infty} 0$ such that, for all $x \in H$:
\begin{equation} 
T (x) = \sum_{n = 1}^\infty s_n \, \langle x \mid v_n \rangle \, u_n \, .
\end{equation} 
This representation $T = \sum_{n=1}^\infty s_n \, u_n \otimes v_n$ is called the Schmidt decomposition of $T$ and the numbers $s_n = s_n (T)$ the 
\emph{singular numbers} of $T$. They are actually the eigenvalues of $|T| = \sqrt{T^\ast T}$ rearranged in non-increasing order. In particular $s_1 = \| T\|$. 

These numbers have the important ``ideal property'':
\begin{displaymath}
s_n (A T B) \leq \| A \| \, s_n (T) \, \| B \| \, .
\end{displaymath}

It is known (see \cite[p.~155]{Carl-Stephani}) that, for all $n \geq 1$, we have $s_n (T) = a_n (T)$, the $n$th \emph{approximation number} of $T$ defined as:
\begin{equation} 
a_n (T) = \inf_{{\rm rank}\, R < n} \| T - R \| \, .
\end{equation} 

For $p > 0$, the Schatten class $S_p (H)$ is the set of all compact operators $T \colon H \to H$ for which 
$\| T \|_p^p := \sum_{n = 1}^\infty [a_n (T)]^p < \infty$. 
\smallskip

We have $S_q (H) \subseteq S_p (H)$ for $0 < q \leq p$. 
\medskip

For composition operators $C_\phi$, D. Luecking  (\cite[Corollary~2]{Luecking}) characterized their membership in the Schatten classes. 

For $\gamma > - 1$, let $dA_\gamma (z) = (\gamma  +1) (1 - |z|^2)^\gamma \, dA (z)$. For $\gamma = - 1$, we set $H^2 = {\mathfrak B}_{- 1}^2$ and 
$dm = dA_{- 1}$. Then, for $\gamma \geq - 1$, the composition operator $C_\phi$ belongs to $S_p ({\mathfrak B}_\gamma^2)$ if and only if:
\begin{equation} 
\sum_{n = 1}^\infty  \sum_{j = 0}^{2^n - 1}\big( 2^{n (\gamma + 2)} A_{\gamma, \phi} (R_{n, j} ) \big)^{p / 2} < \infty \, ,
\end{equation} 
where $A_{\gamma, \phi}$ is the pull-back measure of $A_\gamma$ by $\phi$ (by $\phi^\ast$ for $\gamma = - 1$).

\bigskip

As usual, the notation $A \lesssim B$ means that $A \leq C \, B$ for some positive constant $C$,  which may depend on some parameters, and $A \approx B$ 
means that $A \lesssim B$ and $A \gtrsim B$.

\section {Comparison of approximation numbers} \label{comparison}

\subsection{Main result} 

In the introduction,  we said that, in some sense, $a_n^{{\mathcal D}^{\, 2}} (C_\phi)$ is ``greater'' than $a_n^{H^2} (C_\phi)$, which is ``greater'' 
than $a_n^{{\mathfrak B}^2} (C_\phi)$. This vague statement is made more precise in the following result.

\begin{theorem} \label{main result} 
For any symbol $\phi$, we have, for every $n \geq 1$:
\begin{equation} 
\prod_{j = 1}^n a_j^{{\mathfrak B}^2} (C_\phi) \leq \prod_{j = 1}^n a_j^{H^2} (C_\phi) \leq \prod_{j = 1}^n a_j^{{\mathcal D}^{\, 2}} (C_\phi) \, .
\end{equation} 
\end{theorem} 

It is understood that if $C_\phi$ is not bounded on ${\mathcal D}^{\, 2}$, then $a_j^{{\mathcal D}^{\, 2}} (C_\phi) = + \infty$.
\smallskip\goodbreak

As a consequence, we recover a previous result (\cite[Corollary~3.2]{LLQR-JFA2013}; see also \cite[Theorem~2.5]{IsabelleII}).
\goodbreak

\begin{corollary}
For any symbol $\phi$, we have:
\begin{itemize}
\setlength\itemsep {-0.1 em}
\item [$1)$ $a)$] if $C_\phi$ is compact on ${\mathcal D}^{\, 2}$, then $C_\phi$ is compact on $H^2$; 

\item [\phantom{$1)$} $b)$] if $C_\phi$ is compact on $H^2$, then $C_\phi$ is  compact on ${\mathfrak B}^2$.
\end{itemize}
Moreover, for every $p > 0$, we have:
\begin{itemize}
\setlength\itemsep {-0.1 em}
\item [$2)$ $a)$] if $C_\phi \in S_p ({\mathcal D}^{\, 2})$, then $C_\phi \in S_p (H^2)$;

\item [\phantom{$2)$} $b)$] if $C_\phi \in S_p (H^2)$, then $C_\phi \in S_p ({\mathfrak B}^2)$.
\end{itemize}
\end{corollary}

Items $1)$~$a)$ and $2)$~$a)$ are not sharp, since we proved in \cite[Theorem~2.9]{LLQR-JFA2013} that if $C_\phi$ is compact on ${\mathcal D}^{\, 2}$, 
then $C_\phi$ belongs to all Schatten classes $S_p (H^2)$, with $p > 0$. However, we will see in Section~\ref{Hardy-Bergman} that the item $1)$~$b)$ is sharp,  
but $2)$~$b)$ is not. 
\smallskip

We will prove below these results in a more general setting in Theorem~\ref{theo compos op}.

\subsection {Subordination of sequences}

Let ${\mathcal S}$ be the set of non-increasing sequences $u = (u_j)_{j \geq 1}$ of real numbers. If $u, v \in {\mathcal S}$, the sequence $u$ is said to be 
\emph{subordinate} to the sequence $v$, and we write $u \prec v$, if:
\begin{equation} \label{subordin}
\qquad \qquad \sum_{j = 1}^n u_j \leq \sum_{j = 1}^n v_j  \qquad \text{for all } n \geq 1 \, .
\end{equation} 

For example, if $u = (1, 1, 0, 0, \ldots)$ and $v = (2, 0, 0, \ldots)$, we have $u \prec v$.
\smallskip

We have this basic stability property of this notion (see \cite[Theorem~1.16, p.~13]{Simon}).

\begin{proposition} \label{image subordin}
Let $I$ be an interval of $\R$ and $h \colon I \to \R$ be increasing and convex. Then, if $u, v \in {\mathcal S}$ are sequences of numbers in $I$, 
we have:
\begin{displaymath} 
u \prec v \quad \Longrightarrow \quad h (u) \prec h (v) \, .
\end{displaymath} 
\end{proposition}
\begin{proof}
We recall a two-lines proof. We may assume that $h$ is ${\cal C}^2$. We fix $n \geq 1$ and set $a = \min \{u_n, v_n \}$. Then, for $x \in I$ and $x > a$:
\begin{displaymath} 
h (x) = h (a) + (x - a) h ' (a) + \int_a^{+ \infty} (x - t)^+ h'' (t) \, dt \, .
\end{displaymath} 
One easily checks, using \eqref{subordin}, that $\sum_{j = 1}^n (u_j - t)^+ \leq \sum_{j = 1}^n (v_j - t)^+$ for all $t \geq a$. Hence, thanks to the positivity 
of $h' (a)$ and $h''$:
\begin{align*}
\sum_{j = 1}^n h (u_j) 
& = n h (a) + h ' (a) \sum_{j = 1}^n (u_j - a) + \int_a^{+ \infty} \sum_{j = 1}^n (u_j - t)^+ h '' (t) \, dt \\
& \leq n h (a) + h ' (a) \sum_{j = 1}^n (v_j - a) + \int_a^{+ \infty} \sum_{j = 1}^n (v_j - t)^+ h '' (t) \, dt \\
& = \sum_{j = 1}^n h (v_j) \, . \qedhere
\end{align*}
\end{proof}

A stronger notion is that of log-subordination.
\begin{definition}
We say that the sequence $u \in {\mathcal S}$ of positive numbers is \emph{log-subordinate} to the sequence $v \in {\mathcal S}$ of positive numbers 
if $\log u \prec \log v$. In other terms, if:
\begin{displaymath} 
\qquad\qquad \prod_{j = 1}^n u_j \leq \prod_{j = 1}^n v_j \qquad \text{for all } n \geq 1 \, .
\end{displaymath} 
\end{definition}

The following result will be useful.

\begin{proposition} \label{log-subordin}
For sequences of positive numbers $u, v \in {\mathcal S}$, the following two conditions are equivalent:
\begin{itemize}
\setlength\itemsep {-0.1 em}
\item [$1)$] $\log u \prec \log v$;

\item [$2)$] $u^p \prec v^p$ for all $p > 0$.
\end{itemize}
\end{proposition}
\begin{proof}
If $\log u \prec \log v$, it suffices to apply Proposition~\ref{image subordin} to the sequences $\log u$ and $\log v$ and to the function $h (x) = \e^{p x}$ 
to get $u^p \prec v^p$. 

Conversely, if $u^p \prec v^p$ for all $p > 0$, we have:
\begin{displaymath} 
\bigg( \frac{1}{n} \sum_{j = 1}^n u_j^p \bigg)^{1 / p} \leq \bigg( \frac{1}{n} \sum_{j = 1}^n v_j^p \bigg)^{1 / p} \, ,
\end{displaymath} 
and letting $p$ going to $0$, we get:
\begin{displaymath} 
\bigg( \prod_{j = 1}^n u_j \bigg)^{1 /n} \leq \bigg( \prod_{j = 1}^n v_j \bigg)^{1 /n} \, ,
\end{displaymath} 
i.e. $\log u \prec \log v$.
\end{proof}
\begin{corollary} \label{coro log-subordin}
Let $u, v \in {\mathcal S}$ be two sequences of positive numbers such that $u$ is log-subordinate to $v$. Then for $N \geq n$:
\begin{equation} \label{eq 1 log-sub}
u_N \leq v_1^{n / N} v_n^{1 - n / N} \, .
\end{equation} 
In particular, for any $n \geq 1$:
\begin{equation} \label{eq 2 log-sub}
u_{2 n} \leq \sqrt{v_1 v_n} \, .
\end{equation} 
\end{corollary}
\begin{proof}
We have:
\begin{displaymath} 
u_N^N \leq \prod_{j = 1}^N u_j \leq \prod_{j = 1}^N v_j = \prod_{j = 1}^n v_j \prod_{j = n + 1}^N v_j \leq v_1^n v_n^{N - n} \, ,
\end{displaymath} 
and \eqref{eq 1 log-sub} follows. Now, the choice $N = 2 n$ gives \eqref{eq 2 log-sub}.
\end{proof}

Note that the choice $N = [n \log n]$ can be useful (see \cite{Bayart-Queff-Seip}).

\subsection {Singular numbers} 

The following Weyl type result is crucial for the proof of our main result. It is certainly known by specialists, but we have not found any reference.

\begin{proposition} \label{s_1, ..., s_n}
Let $T$ be a compact operator on a separable complex Hilbert space $H$ and $T = \sum_{j = 1}^\infty s_j \, u_j \otimes v_j$ its Schmidt decomposition. 
Then, for every integer $n \geq 1$:
\begin{displaymath}
s_1 \cdots s_n = \max \big| \det \big( \langle T f_j \mid g_i \rangle \big)_{i, j} \big| \, ,
\end{displaymath}
where the supremum is taken over all pairs $(f_j)_{1 \leq j \leq n}$ and $(g_i)_{1 \leq i \leq n}$ of orthonormal systems of length $n$ in $H$.
\end{proposition}
\begin{proof}
First, assume that $H$ is $n$-dimensional. We may assume that $H = \ell_2^n$ and we denote $(e_i)_{1 \leq i \leq n}$ its canonical basis. 

Since $T (v_j) = s_j u_j$, we have $\det \big( \langle T v_j \mid u_i \rangle \big)_{i, j} =s_1\cdots s_n$.

Now, if $(f_j)_{1 \leq j \leq n}$ and $(g_i)_{1 \leq i \leq n}$ are two orthonormal systems, we consider the following diagram:
\begin{displaymath}
\ell_2^n \converge^{U} \ell_2^n \converge^{T} \ell_2^n \converge^{V} \ell_2^n 
\end{displaymath}
where $U$, $V$ are the unitary operators defined by:
\begin{displaymath}
U \Big( \sum_{j = 1}^n t_j e_j \Big) = \sum_{j = 1}^n t_j f_j \quad \text{and} \quad V x = \sum_{j = 1}^n \langle x \mid g_j \rangle \, e_j \, .
\end{displaymath}
We observe that
\begin{displaymath}
\langle V T U e_j \mid e_i \rangle = \langle T U e_j \mid V^\ast e_i \rangle = \langle T f_j \mid g_i \rangle \, ,
\end{displaymath}
so that:
\begin{displaymath}
\big| \det \big( \langle T f_j \mid g_i \rangle \big)_{i, j} \big| = | \det V | \, |\det T | \, |\det U| = |\det T | = s_1 \cdots s_n \, .
\end{displaymath}

In the general case, denote by $P_n$ and $Q_n$ the orthogonal projections onto $F_n := [f_1, \ldots , f_n]$ and $G_n := [g_1, \ldots, g_n]$ respectively. 
We can see $F_n$ and $G_n$ as isometric copies of $\ell_2^n$. Observe that $\langle T f_j \mid g_i\rangle = \langle Q_n T P_n f_j \mid g_i\rangle$. By the 
above special case, we get, using the ideal property of singular numbers:
\begin{displaymath}
\big| \det \big( \langle T f_j \mid g_i \rangle \big)_{i, j} \big| = \prod_{j = 1}^n s_j (Q_n T P_n) \leq \prod_{j = 1}^n s_j (T) \, . \qedhere
\end{displaymath}
\end{proof}
%

\subsection {Comparison principle for operators} \label{subsection operators}

For convenience, we say that an operator $U \colon H \to K$ between Hilbert spaces is \emph{unitary} if it is a surjective isometry, even if $H \neq K$.
\smallskip

V.~\`E~Kacnel'son (\cite{Kacnelson}) proved the following result.

\begin{theorem} [V.~\`E~Kacnel'son] \label{theo Kacnelson} 
Let $H$ be a separable complex Hilbert space and $(e_i)_{i \geq 0}$ a fixed orthonormal basis of $H$. 
Let $A \colon H \to H$ be a bounded linear operator. We assume that the matrix of $A$ with respect to this basis is lower-triangular: 
$\langle A e_j \mid e_i \rangle = 0$ for $i < j$.

Let $(d_j)_{j \geq 0}$ be an increasing sequence of positive real numbers and $D$ the (possibly unbounded) diagonal operator such that 
$D (e_j) = d_j e_j$, $j \geq 0$. Then the operator $D^{ - 1} A D \colon H \to H$ is bounded and moreover:
\begin{equation} \label{ineq Kacnelson}
\| D^{- 1} A D \| \leq \| A \| \, .
\end{equation}
\end{theorem}

In \cite{Isabelle}, this theorem was extended in the framework of Banach spaces with $1$-unconditional basis and used for the study of composition operators, 
and in \cite{IsabelleII} to compare the Schatten-class norms of weighted Hilbert spaces of analytic functions. 
\smallskip

We have the following generalization (the case $n = 1$ giving $\| D^{- 1} A D \| \leq \| A \|$).

\begin{theorem} \label{ext theo Kacnelson} 
With the notation of Theorem~\ref{theo Kacnelson}, and assuming moreover that $A$ is compact, we have, for every $n \geq 1$:
\begin{equation}
\prod_{j = 1}^n s_j (D^{- 1} A D) \leq \prod_{j = 1}^n s_j (A) \, .
\end{equation}
\end{theorem}

In other words, the sequence $\big( s_j ( D^{- 1} A D) \big)_j$ is log-subordinate to $\big( s_j (A) \big)_j$. 

\begin{proof}
Let $\C_0$ be the right-half plane $\C_0 = \{z \in \C \tq \Re z > 0 \}$ and $H_N = {\rm span}\, \{ e_j \tq j \leq N \}$. We set:
\begin{displaymath}
a_{i, j} = \langle A e_j \mid e_i \rangle 
\end{displaymath}
and 
\begin{displaymath}
A_N = P_N A P_N \, ,
\end{displaymath}
where $P_N$ is the orthogonal projection from $H$ into $H$ with range $H_N$. 
We consider, for $z \in \overline{\C_0}$:
\begin{displaymath}
A_N (z) = D^{- z} A_N D^z \colon H \to H \, ,
\end{displaymath}
where $D^z (e_n) = d_n^z e_n$. 

If $\big( a_{i, j}^N (z) \big)_{i, j}$ is the matrix of $A_N (z)$ on the basis $\{e_j \tq j \geq 0 \}$ of $H$, we clearly have:
\begin{displaymath}
a_{i, j}^N (z) = \left\{
\begin{array}{ll}
a_{i, j} (d_j / d_i)^z & \text{ if } i , j \leq N \smallskip\\
0 & \text{ otherwise}  \, .
\end{array}
\right.
\end{displaymath}
In particular, we have, by hypothesis:
\begin{displaymath}
\qquad\qquad\quad | a_{i, j}^N (z) | \leq \sup_{k, l} |a_{k, l}| := M \, , \qquad \text{for all } z \in \overline{\C_0} \, .
\end{displaymath}
Since $\| A_N (z) \|^2 \leq \| A_N (z)\|_{HS}^2 = \sum_{i, j \leq N} |a_{i, j}^N (z)|^2 \leq (N + 1)^2 M^2$, we get:
\begin{displaymath}
\qquad \qquad \| A_N (z) \| \leq (N + 1) M \qquad \text{for all } z \in \overline{\C_0} \, .
\end{displaymath}

Let us consider the function $u \colon \overline{\C_0} \to \overline{\C_0}$ defined by:
\begin{equation}
u (z) = \prod_{j = 1}^n s_j \big( A_N (z) \big) \, .
\end{equation}
This function $u$ is continuous on $\overline{\C_0}$. 

If $\alpha$ denotes a pair $(f_j)$, $(g_i)$ of orthonormal systems of length $n$ of $H$, we set, for $z \in \overline{\C_0}$:
\begin{displaymath}
F_\alpha (z) = \det \big( \langle A_N (z) f_j \mid g_i \rangle \big)_{i, j} \, ,
\end{displaymath}
the function $F_\alpha$ is analytic in $\C_0$ and continuous on $\overline{\C_0}$. By Proposition~\ref{s_1, ..., s_n}, we have 
$u = \sup_\alpha |F_\alpha|$, so that $u$ is subharmonic in $\C_0$. Moreover:
\begin{displaymath}
\qquad\qquad\quad u (z) \leq \| A_N (z)\|^n \leq [(N + 1) M]^n \qquad \text{for } z \in \overline{\C_0} \, ,
\end{displaymath}
and:
\begin{displaymath}
\qquad\qquad\qquad u (z) = \prod_{j = 1}^n s_j (A_N) \leq \prod_{j = 1}^n s_j (A) \qquad \text{for } z \in \partial{\C_0} \, ,
\end{displaymath}
since the operator $D^z \colon H \to H$ is then unitary. Hence we can use the following form of the maximum principle.

\begin{theorem} [maximum principle] 
Let $\Omega$ be an arbitrary domain in $\C$, with $\Omega \neq \C$, and $u \colon \overline{\Omega} \to \R$ a function subharmonic in $\Omega$, 
and continuous and bounded above on $\overline{\Omega}$. Then:
\begin{displaymath}
\sup_{\overline \Omega} u = \sup_{\partial \Omega} u \, .
\end{displaymath}
\end{theorem}
\noindent
This theorem is proved in \cite[Theorem~15.1, p.~190]{Bak-Newman} for $u = |f|$, with $f \colon \overline \Omega \to \C$ holomorphic in $\Omega$, 
and continuous and bounded on $\overline \Omega$, and in \cite[Theorem~5.16, p.~232]{Hayman}. It follows that:
\begin{displaymath}
\sup_{\Re z \geq 0} u (z) \leq \prod_{j = 1}^n s_j (A) \, . 
\end{displaymath}
In particular $u (1) \leq \prod_{j = 1}^n s_j (A)$, or else:
\begin{displaymath}
\prod_{j = 1}^n s_j (D^{- 1} A_N D) \leq \prod_{j = 1}^n s_j (A) \, .
\end{displaymath}

Now, since the matrix of $A - A_N$ is lower-triangular, the inequality \eqref{ineq Kacnelson}, applied to $A - A_N$, gives 
$\| D^{- 1} (A - A_N) D\| \leq \| A - A_N \| \converge_{N \to \infty} 0$. Moreover, for each $j \geq 1$, the map $T \in {\mathcal L} (H) \mapsto s_j (T)$ 
is continuous, since $|s_j (T_1) - s_j (T_2) | \leq \| T_1 - T_2 \|$. Then, letting $N$ tend to infinity, we obtain that 
\begin{displaymath} 
s_j (D^{- 1} A_N D) \converge_{N \to \infty} s_j (D^{- 1} A D) \, , 
\end{displaymath} 
and the result follows. 
\end{proof}

An alternative proof of Theorem~\ref{ext theo Kacnelson} can be given using antisymmetric tensor products.
\begin{proof}[Alternative proof of Theorem~\ref{ext theo Kacnelson}] 
Let $I$ denote the set of all increasing $n$-tuples $\alpha = (i_1 < i_2 < \cdots < i_n)$ of non-negative integers. Let  $(u_\alpha)_{\alpha \in I}$ be the 
orthonormal basis of $\Lambda^{n}(H)$, the $n$-th exterior power of $H$, defined by:
\begin{displaymath} 
\qquad\qquad u_\alpha = e_{i_1} \wedge e_{i_2} \wedge \cdots \wedge e_{i_n} \, , \quad \alpha \in I \, .
\end{displaymath} 
We use the general fact that:
\begin{displaymath} 
\prod_{j = 1}^n s_j (D^{- 1} A D) = \| \Lambda^{n} (D^{- 1} A D) \| \, ,
\end{displaymath} 
where $\Lambda^n$ denotes the $n$-th skew product. 

Since $\Lambda^n (U V) = \Lambda^n (U)  \Lambda^n(V)$ (\cite[page~10]{Simon}), we get:
\begin{align*} 
\Lambda^n (D^{- 1} A D) 
& = \Lambda^n (D^{- 1})  \Lambda^n (A)  \Lambda^n (D) = \big[ \Lambda^n (D)\big]^{- 1} \Lambda^n (A) \Lambda^n (D) \\ 
& =: \Delta^{- 1} \Lambda^n (A) \Delta \, , 
\end{align*} 
where $\Delta$ is the diagonal operator on the basis $(u_\alpha)$ with diagonal elements $\delta_\alpha = d_{i_1}\cdots d_{i_n}$ if 
$\alpha = (i_1 < i_2 < \cdots < i_n)$.

Now, we claim that  $\Lambda^n (A)$ is lower triangular in the following sense. If $\alpha = (i_1 < i_2 < \cdots < i_n)$ and 
$\beta = (j_1 < j_2 < \cdots < j_n)$  are two elements of $I$, then:
 \begin{equation}\label{elof} 
\delta_\alpha < \delta_\beta \Longrightarrow \big\langle \Lambda^n (A) \, u_\beta , u_\alpha \big\rangle = 0 \, .
\end{equation} 
Indeed, assume that $\big\langle \Lambda^n (A) \, u_\beta , u_\alpha \big\rangle \neq 0$. Since:
\begin{align*}
\big\langle \Lambda^n (A) \, u_\beta , u_\alpha \big\rangle 
& = \big\langle A e_{j_1}\wedge A e_{j_1} \cdots \wedge A e_{j_n} ,  e_{i_1} \wedge e_{i_1} \cdots \wedge e_{i_n} \big\rangle \\
& = \det \big( \langle A e_{j_p}, e_{i_q} \rangle \big)_{1 \leq p, q \leq n} \, ,
\end{align*}
it follows, by definition of determinants, that there exists a permutation $\sigma$ of $\{1, 2, \ldots, n\}$ such that:
\begin{displaymath} 
\prod_{1 \leq k \leq n} \langle A e_{j_k}, e_{i_{\sigma(k)}} \rangle \neq 0 \, , 
\end{displaymath} 
implying that $\ i_{\sigma(k)}\geq j_k$ for each $k$. But then, since $l \mapsto d_l$ is nondecreasing:
\begin{displaymath} 
\delta_\alpha = \prod_{1 \leq k \leq n} d_{i_{\sigma(k)}} \geq \prod_{1 \leq k \leq n} d_{j_k} = \delta_\beta \, .
\end{displaymath} 

Now, \eqref{elof} allows to apply Theorem~\ref{theo Kacnelson} to get the result.
\end{proof}

\noindent {\bf Remark.} We could also remark that the function:
\begin{displaymath} 
u (z) = \prod_{1 \leq j \leq n} s_j (D^{- z} A_N D^z) = \| \Lambda^n (D^{- z} A_N D^z) \|
\end{displaymath} 
 is subharmonic since it is a norm, on $\Lambda^n (H)$, and hence a supremum of moduli of  the holomorphic functions 
$z \mapsto l (D^{- z} A_N D^z)$, for $l$ a linear functional on $\Lambda^n (H)$.
\smallskip

\begin{corollary}
With the notation of Theorem~\ref{theo Kacnelson}, $D^{- 1} A D$ is compact if $A$ is. Moreover, for any $p > 0$, if $A \in S_p (H)$, so does $D^{- 1} A D$, 
and:
\begin{displaymath}
\| D^{- 1} A D \|_p \leq \| A \|_p \, .
\end{displaymath}
\end{corollary}
\begin{proof}
Since $\big( s_n (D^{- 1} A D) \big)_n$ is log-subordinate to $\big( s_n (A) \big)_n\,$, Corollary~\ref{coro log-subordin} gives the first assertion, and 
Proposition~\ref{log-subordin} gives the second one.
\end{proof}
%
\subsection {Application to composition operators} 
\smallskip

We consider here general weighted Hilbert spaces of analytic functions on $\D$. 
\smallskip

Let $\beta = (\beta_k)_{k \geq 0}$ be a sequence of positive numbers such that:
\begin{equation} \label{cond sur beta}
\liminf_{k \to \infty} \beta_k^{\ 1/k} \geq 1
\end{equation}
(as we will see right after, this condition ensure that the evaluation maps are bounded) and let $H^{2}(\beta)$ be the Hilbert space of functions 
$f (z) = \sum_{k = 0}^\infty c_k z^k$ such that:
\begin{equation} \label{H2 (beta)}
\Vert f \Vert_{H^2 (\beta)}^2 := \sum_{k = 0}^\infty \beta_k \, |c_k|^2 <\infty \, .
\end{equation}
This is a Hilbert space of analytic functions on $\D$ with a reproducing kernel $K_a$, namely:
\begin{equation}
\qquad \qquad f (a) = \langle f, K_a \rangle \quad \text{for all } f \in H^{2}(\beta) \, , 
\end{equation}
because the evaluations $f \in H^{2}(\beta) \mapsto f (a)$ are continuous:
\begin{displaymath}
\bigg| \sum_{k = 0}^\infty c_k a^k \bigg| 
\leq \bigg( \sum_{k = 0}^\infty \beta_k \, |c_k|^2\bigg)^{1/2} \bigg( \sum_{k = 0}^\infty \beta_k^{ - 1} |a|^{2 k} \bigg)^{1/2} 
< \infty \, ,
\end{displaymath}
thanks to condition \eqref{cond sur beta}. 
\smallskip

The canonical orthonormal basis of $H^2 (\beta)$ is formed by the normalized monomials 
\begin{equation} \label{mono} 
\qquad e_k^{\beta} (z) = \frac{z^k}{\sqrt{\beta_k}} \, \raise 1 pt \hbox{,} \quad k = 0, 1, \ldots \, ; 
\end{equation}
so we have, for all $a \in \D$:
\begin{equation} 
\| K_a \|_{H_\omega^2}^2 = \sum_{n = 0}^\infty |e_n^\beta (a)|^2 = \sum_{n = 0}^\infty \frac{1}{\beta_n} \, |a|^{2 n} \, .
\end{equation} 

We refer to \cite{Co-MacC} or \cite{Zhu} for more on those spaces. See also \cite{Karim-Pascal} for an alternative definition. 
\smallskip

For example, the weighted Dirichlet space ${\mathcal D}^{\, 2}_\alpha$ corresponds to $\beta_k \approx (k + 1)^{1 - \alpha}$. In particular, 
the Hardy space $H^2$ corresponds to $\beta_k = 1$, the Bergman space ${\mathfrak B}^2$ to $\beta_k = 1 / (k + 1)$, and the Dirichlet space 
${\mathcal D}^{\, 2}$ to $\beta_k = (k + 1)$. 
\smallskip

For the weights 
\begin{displaymath} 
\beta_k = \frac{(k + 1)! \, \Gamma (\alpha + 2)}{\Gamma (k + \alpha + 1)} \, \raise 1 pt \hbox{,} 
\end{displaymath} 
we get, using the binomial formula $\sum_{k = 0}^\infty \frac{\Gamma (k + \alpha)}{k! \, \Gamma (\alpha)} \, x^k = (1 - x)^{- \alpha}$ for $|x| < 1$, 
that the reproducing kernels are, for $a \neq 0$:
\begin{align}
\qquad \qquad K_a^\alpha (z) & = \frac{1}{\alpha (\alpha + 1)} \, \frac{(1 - \bar{a} z)^{- \alpha} - 1}{\bar{a} z} 
\, \raise 1 pt \hbox{,} \qquad \text{for } \alpha > 0 \, ; \\
\qquad \qquad K_a^0 (z) & = \frac{1}{\bar{a} z} \, \log \frac{1}{(1 - \bar{a}z)} \, \raise 1 pt \hbox{,}
\end{align}
(with $K_0^\alpha (z) = 1 / ( \alpha + 1)$ and $K_0^0 (z) = 1$). 
\medskip

Let us point out that $\lim_{\alpha \to 0^+} K_a^\alpha (z) = K_a^0 (z)$. 
\smallskip\goodbreak

Let now $\varphi \colon \D \to \D$ be an analytic map. We assume that:
\begin{equation} \label{phi(0)=0} 
\varphi (0) = 0 \, .
\end{equation}
This map $\varphi$ induces \emph{formally} a lower-triangular composition operator $C_\varphi$ on $H^2 (\beta)$ since: 
\begin{displaymath}
\langle C_{\varphi} ( e_j^\beta),  e_i^\beta \rangle = \frac{1}{\sqrt{\beta_i \beta_j}} \, \langle \varphi^j, z^i \rangle = 0 \quad \text{for } i < j \, .
\end{displaymath}

\noindent {\bf Remark.} We can often omit condition \eqref{phi(0)=0}. In fact, let us consider the  automorphisms $\varphi_a \colon \D \to \D$, $a \in \D$, 
given by  $\varphi_{a} (z) = \frac{a - z}{1 - \bar{a} z}$. When $C_{\phi_a}$ is \emph{bounded on $H^2 (\beta)$}, then, with $a = \varphi (0)$, 
the function $\psi = \varphi_a \circ \varphi$ satisfies $\psi (0) = 0$ and $\phi = \phi_a \circ \psi$; hence $C_\phi = C_\psi \circ C_{\phi_a}$ and 
$C_\psi = C_\phi \circ C_{\phi_a}$, so:
\begin{displaymath} 
 \| C_{\phi_a}\|^{- 1} a_n (C_\psi) \leq a_n (C_\varphi) \leq  \| C_{\phi_a}\|\,  a_{n} (C_\psi) \, .
\end{displaymath} 

A necessary condition for having $C_{\phi_a}$ bounded on $H^2 (\beta)$ is that $\phi_a \in H^2 (\beta)$. Since 
$\phi_a (z) = a + \sum_{k = 1}^\infty \bar{a}^{k - 1} (|a|^2 - 1) z^k$, we have $\phi_a \in H^2 (\beta)$ for all $a \in \D$ if 
$\lim_{k \to \infty} \beta_k^{\ 1 /k} = 1$. 
\smallskip

For weighted Dirichlet spaces ${\mathcal D}^{\, 2}_\alpha$, with any $\alpha > - 1$, the automorphisms $\varphi_a$ define bounded composition 
operators on ${\mathcal D}^{\, 2}_\alpha$. In fact, we have, for $f \in {\mathcal D}^{\, 2}_\alpha$:
\begin{align*}
\| f \circ \phi_a \|_{{\mathcal D}^{\, 2}_\alpha}^2 
& = |f (a)|^2 + (\alpha + 1) \int_\D |f ' [\phi_a (z)] |^2 |\phi_a ' (z)|^2 (1 - |z|^2)^\alpha \, dA (z) \\
& = |f (a)|^2 + (\alpha + 1) \int_\D | f ' (w)|^2 (1 - |\phi_a (w)|^2)^\alpha \, dA (w) \, .
\end{align*}
Since:
\begin{displaymath}
\frac{1 - |\phi_a (w)|^2}{1 - |w|^2} = \frac{1 - |a|^2}{|1 - \bar{a} w|^2} \, \raise 1 pt \hbox{,}
\end{displaymath}
we have $1 - |\phi_a (w)|^2 \approx 1 - |w|^2$ and we get:
\begin{displaymath}
\| f \circ \phi_a \|_{{\mathcal D}^{\, 2}_\alpha}^2 \lesssim |f (a)|^2 + (\alpha + 1) \int_\D | f ' (w)|^2 (1 - |w|^2)^\alpha \, dA (w) 
\approx \| f \|_{{\mathcal D}^{\, 2}_\alpha}^2 \, .
\end{displaymath}

For $\alpha \geq 0$, that follows directly from \cite[Theorem~1]{Zorb} (see also \cite[Section~6.12]{Shapiro}, 
\cite[Theorem~1.3 and Proposition~3.1]{Karim-Pascal} or \cite[Theorem~3.1]{Pau-Perez}), since $\phi_a$ is univalent. 
\smallskip

We will write for short $C_\phi^\beta$ to designate the operator $C_\phi$ acting on $H^2 (\beta)$.
\medskip

As an application of the general principles of Section~\ref{subsection operators} we have the following result, whose first items were previously obtained 
by I.~Chalendar and J.~Partington in \cite{Isabelle} and \cite{IsabelleII} (actually (3.b) is also proved in \cite{IsabelleII}, but for values $p\ge1$).
\goodbreak

\begin{theorem} \label{theo compos op} 
Let $H^2 (\beta)$ and $H^2 (\gamma)$ be two weighted Hilbert spaces. Assume that $\gamma$ is dominated by $\beta$ in the sense that the sequence 
$(\beta_k / \gamma_k)$ is increasing, so that the continuous inclusion $H^2 (\beta) \subseteq H^2 (\gamma)$ holds. Then, for $\phi \colon \D \to \D$ with 
$\phi (0) = 0$:
\begin{itemize}
\setlength\itemsep {-0.1 em}
\item [$1)$] if $C_\phi^\beta$ is bounded, $C_\phi^\gamma$ is bounded as well, and $\| C_\phi^\gamma \| \leq \| C_\phi^\beta \|$; 

\item [$2)$] if $C_\phi^\beta$ is compact, so is $C_\phi^\gamma$;

\item [$3)$] the sequence $s^\gamma = \big( s_n (C_\phi^\gamma) \big)_{n \geq 1}$ is log-subordinate to the sequence 
$s^\beta = \big( s_n (C_\phi^\beta) \big)_{n \geq 1}$,  so that:
\vskip -1.5 em \null
\begin{itemize}
\item [\phantom{$3)$} $a) $] $s_{2n} (C_\phi^\gamma) \leq \sqrt{s_{1} (C_\phi^\beta)} \,  \sqrt{s_{n} (C_\phi^\beta) }$, for all $n \geq 1$;

\item [\phantom{$3)$} $b)$] $C_\phi^\beta \in S_p \big(H^2 (\beta) \big) \Longrightarrow C_\phi^\gamma \in S_p \big( H^2 (\gamma) \big)$, 
for any $p > 0$.
\end{itemize}
\end{itemize}
\end{theorem}
\noindent {\bf Remark.} Let us mention that we can apply the previous theorem in the framework of weighted Dirichlet spaces. Indeed, let $0 < \beta < \gamma$ 
and consider the two weights: 
\begin{displaymath} 
\beta_k = \frac{k. k! \,\Gamma (\beta + 2)}{\Gamma (k +\beta + 1)} \qquad \text{and} \qquad 
\gamma_k = \frac{k. k!\,\Gamma (\gamma + 2)}{\Gamma (k + \gamma + 1)} 
\end{displaymath} 
associated with the weighted Dirichlet spaces $\mathcal{D}_{\beta}^{\, 2}$ and  $\mathcal{D}_{\gamma}^{\, 2}$ respectively, with $\gamma > \beta$, so 
that $\mathcal{D}_{\beta}^{\, 2}\subset \mathcal{D}_{\gamma}^{\, 2}$ . In order to apply our comparison Theorem~\ref{theo compos op}, we have to 
show that the sequence $(\beta_k/\gamma_k)$ increases. But
\begin{displaymath} 
\frac{\beta_k}{\gamma_k} = \frac{\Gamma (\beta + 2)}{\Gamma (\gamma + 2)}\, \frac{\Gamma (\gamma + k + 1)}{\Gamma (\beta + k + 1)} 
=: \frac{\Gamma (\beta + 2)}{\Gamma (\gamma + 2)}\, A_k \, ,
\end{displaymath} 
and, setting $h =\gamma - \beta > 0$ and $x_k = \beta + k + 1$, we see that:
\begin{displaymath} 
A_k = \frac{\Gamma (x_{k} + h)}{\Gamma (x_k)} \, \cdot
\end{displaymath} 
Since the function $\Gamma$ is log-convex, the map $x \mapsto \frac{\Gamma (x + h)}{\Gamma (x)}$ increases on $(0, \infty)$, and we get that 
the sequence $(\beta_k/\gamma_k)$ increases. 

\begin{proof} [Proof of Theorem~\ref{theo compos op}] 
We  set $d_k = \sqrt{\beta_k / \gamma_k}$ and $e_k (z) = z^k$. 

Let $J \colon H^2 (\beta) \to H^2 (\gamma)$ the unitary (onto isometry) and diagonal operator defined by $J (e_k) = d_k e_k$, for all $k \geq 0$. 

The operator $A = J C_\phi^\beta J^{- 1}$ maps $H^2 (\gamma)$ into itself and $s_n (A) = s_n (C_\phi^\beta)$ for all $n \geq 1$ (in particular 
$\| A \|_{{\cal L} (H^2 (\gamma) )} = \| C_\phi^\beta \|_{{\cal L} (H^2 (\beta) )}$). Moreover, $A$ has a lower-triangular matrix. 

Now we consider the diagonal operator $D \colon H^2 (\gamma) \to H^2 (\gamma)$ defined by $D (e_k) = d_k e_k$. In general, it is an unbounded operator. 
It is plain that $D^{- 1} J \colon H^2 (\beta) \to H^2 (\gamma)$ is the canonical inclusion, since $(D^{- 1} J) (e_k) = e_k$ for all $k \geq 0$. Hence 
$(D^{- 1} J) C_\phi^\beta = C_\phi^\gamma (D^{- 1} J)$, and since $A J = J C_\phi^\beta$, we have the following commutative diagram:
\begin{displaymath} 
\xymatrix{
& H^2 (\beta) \ar[r]^{C_\phi^\beta} \ar[d]_J & H^2 (\beta) \ar[d]^J & \\
H^2 (\gamma) \ar[r]_D & H^2 (\gamma) \ar[r]_A & H^2 (\gamma) \ar[r]_{D^{ - 1}} & H^2 (\gamma)
}
\end{displaymath} 

By Theorem~\ref{ext theo Kacnelson}, we get:
\begin{displaymath}
\log s (D^{- 1} A D) \prec \log s (A) 
\end{displaymath}
(so we have, in particular, 
$\| D^{- 1} A D \|_{{\cal L} ( H^2 (\gamma) )} \leq \| A \|_{{\cal L} ( H^2 (\gamma) )} = \| C_\phi^\beta \|_{{\cal L} ( H^2 (\beta) )}$). 
But $D^{- 1} A D = C_\phi^\gamma$, and this proves Theorem~\ref{theo compos op}, using Proposition~\ref{log-subordin} and 
Corollary~\ref{coro log-subordin}.
\end{proof}
\goodbreak

\noindent{\bf Remark.} Actually, the same proof gives the following generalization of Theorem~\ref{theo compos op}.
\goodbreak

\begin{theorem} \label{theo compos op-bis}
With the hypothesis of Theorem~\ref{theo compos op}, let $T \colon {\mathcal H}ol (\D) \to {\mathcal H}ol (\D)$ be a linear map such that its 
restriction $T_\beta$ to $H^2 (\beta)$ is bounded from  $H^2 (\beta)$ into $H^2 (\beta)$ and has a matrix in the canonical basis of $H^2 (\beta)$ which is 
lower-triangular. Then:
\begin{itemize}
\setlength\itemsep {-0.1 em}
\item [$1)$] $T_\gamma$ is bounded as well, and $\| T_\gamma \| \leq \| T_\beta \|$; 

\item [$2)$] if $T_\beta$ is compact, so is $T_\gamma$;

\item [$3)$] the sequence of singular numbers $s (T_\gamma) = \big( s_n (T_\gamma) \big)_{n \geq 1}$ is log-subordinate to the sequence 
$s (T_\beta) = \big( s_n (T_\beta) \big)_{n \geq 1}$,  so that:
\vskip -1.2 em \null
\begin{itemize}
\item [\phantom{$3)$} $a) $] $s_{2n} (T_\gamma) \leq \sqrt{s_{1} (T_\beta)} \, \sqrt{s_{n} (T_\beta)}$, for all $n \geq 1$;

\item [\phantom{$3)$} $b)$] $T_\beta \in S_p \big(H^2 (\beta) \big) \Longrightarrow T_\gamma \in S_p \big( H^2 (\gamma) \big)$, for any $p > 0$.
\end{itemize}
\end{itemize}
\end{theorem} 
%

\subsection{Application to conditional multipliers} \label{sec: conditional multipliers}

We first recall the following well-known proposition (and give a short proof, for sake of completeness). Note that this is not shared by the 
Dirichlet spaces $\mathcal{D}^{\, 2}_\alpha$ when $\alpha \leq 0$ (\cite[Theorem~10]{Taylor}; see also \cite[Theorem~2.7]{Stegenga}, 
\cite[Theorem~A]{Kerman-Sawyer}, and \cite[Theorem~4.2]{Wu}). Recall that it is well-known that the space ${\mathcal M} (H^2)$ of multipliers of 
$H^2$ is isometric to $H^\infty$.
\goodbreak

\begin{proposition} \label{multipliers}
For every $\gamma > - 1$, the space $\mathcal{M} (\mathfrak{B}^2_\gamma)$ of multipliers of $\mathfrak{B}^2_\gamma$ is isometric to the space 
$H^\infty$.\par

If $H$ is a Hilbert space of analytic functions on $\D$, containing the constants, and with reproducing kernels $K_a$, $a \in \D$, then the 
space $\mathcal{M} (H)$ of multipliers of $H$ is contained contractively into the space $H^\infty$.
\end{proposition}
\begin{proof} 
If $h f \in H$ for all $f \in H$, then, taking $f = \ind$, we have $h \in H$, so $h$ is analytic. The same proof as in \cite[Proposition~3.1]{Attele} shows that 
$h \in H^\infty$. For sake of completeness we give a short different proof. \par

In fact, we have, for all $a \in \D$:
\begin{equation} \label{adve} 
\qquad M_h^\ast (K_a) = \overline{h (a)} \, K_a \quad \text{for all } a \in \D \, ; 
\end{equation} 
hence $|h (a)| \, \| K_a\| \leq \| M_h^\ast\| \, \| K_a\|$, and, since $\| K_a\|$ is not null, 
that proves that $h \in H^\infty$ and $\| h \|_\infty \leq \| M_h \|$. 

Hence $\mathcal{M} (H) \subseteqq H^\infty$, contractively. 
\smallskip

When $H = \mathfrak{B}^2_\gamma$, we have the reverse inclusion. Indeed, for every $h \in H^\infty$, one clearly has $h f \in \mathfrak{B}^2_\gamma$ and 
$\| h f \|_{\mathfrak{B}^2_\gamma} \leq \| h \|_\infty \| f \|_{\mathfrak{B}^2_\gamma}$ for all $f \in \mathfrak{B}^2_\gamma$, so the multiplication 
operator $M_h \colon \mathfrak{B}^2_\gamma \to \mathfrak{B}^2_\gamma$ is bounded with norm $\leq \| h \|_\infty $. 
\end{proof}

Let now $\varphi$ be an analytic self-map of $\D$ and $H = H^{2}(\beta)$ be a weighted Hilbert space of analytic functions on $\D$, with reproducing 
kernel $K_a$, $a\in \D$, on which $C_\phi$ acts boundedly . We denote its multiplier set, respectively multiplier set  conditionally to $\varphi$, by:  
\begin{equation} 
 \mathcal{M} (H)  = \{w \in H \tq  w f \in H \text{ for each } f \in H \} 
\end{equation} 
and
\begin{equation} 
\mathcal{M} (H, \varphi) = \{w \in H \tq w \, (f \circ \varphi) \in H \text{ for all } f \in H \} \, .
\end{equation} 

We have $\mathcal{M} (H) \subseteq {\mathcal M} (H, \phi)$.
\smallskip

The set $\mathcal{M} (H, \varphi)$ plays an important role in the study of weighted composition operators.
\medskip\goodbreak

\begin{definition}
A Hilbert space $H$ of analytic functions on $\D$, containing the constants, and with reproducing kernels $K_a$, $a \in \D$, is said 
\emph{admissible} if: 
\begin{enumerate}
\setlength\itemsep {-0.05 em}

\item [$(i)$]  $H^2$ is continuously embedded in $H$;  

\item [$(ii)$] $\mathcal{M} (H) = H^\infty$; 

\item [$(iii)$] the automorphisms of $\D$ induce bounded composition operators on $H$;

\item [$(iv)$] $\displaystyle \frac{\Vert K_a\Vert_H}{\Vert K_b\Vert_H} \leq h \bigg(\frac{1 - |b|}{1 - |a|} \bigg)$ for $a, b \in \D$ close to 
$\partial \D$, where $h \colon \R^+\to \R^+$ is an non-decreasing function.

\end{enumerate}
\end{definition}

Note that $(i)$ implies that $\Vert f \Vert_H \leq C \, \Vert f \Vert_{H^2}$ for all $f \in H^2$, for some positive constant $C$, and so ($B_H$ and $B_{H^2}$ 
being the unit ball of $H$ and $H^2$ respectively):
\begin{displaymath} 
\Vert K_{a}\Vert_{H} = \sup_{f \in B_H} |f (a)| \geq C^{- 1} \sup_{f \in B_{H^2}} |f (a) | = C^{- 1} (1 - |a|^2)^{- 1/2} \, ,
\end{displaymath} 
implying that:
\begin{displaymath} 
\lim_{|a|\to 1^-}\Vert K_a\Vert_H = \infty \, .
\end{displaymath} 

\noindent \emph{Examples.} 
\smallskip

1) The weighted Bergman space ${\mathfrak B}_\gamma^2$, with $\gamma > - 1$ is admissible. 

Indeed,  we know that it is continuously embedded in $H^2 = {\mathfrak B}_{- 1}^2$; condition $(ii)$ is Proposition~\ref{multipliers}; 
condition $(iii)$ is  satisfied according to the Remark before Theorem~\ref{theo compos op}, and  
$\Vert K_a \Vert^2 = \frac{1}{(1 - |a|^2)^{\gamma + 2}}\,$, giving $(iv)$. 
\smallskip

2) More generally, we have the following result.
\begin{proposition} \label{prop admissible}
For any decreasing sequence $\beta$ such that the automorphisms of $\D$ induce bounded composition operators on $H^2 (\beta)$, the space $H^2 (\beta)$ is 
admissible.
\end{proposition}

Recall that $H^2 (\beta)$ is defined in \eqref{H2 (beta)}. A particular case is obtained as follows.
Let $\omega \colon (0, 1) \to \R_+$ be an integrable function such that, for some positive and locally bounded function $\rho\colon \R_+\to \R_+$, we have: 
\begin{equation} \label{andalso} 
\qquad \qquad \qquad \qquad \frac{\omega (y)}{\omega (x)} \leq \rho \bigg( \frac{y}{x} \bigg) \quad \text{ for all } x, y \in (0, 1) \, ,
\end{equation}
and let $H^2_\omega$ be the space of analytic functions $f \colon \D \to \C$ such that:
\begin{equation} 
\Vert f \Vert_{H^2_\omega}^{2} := \int_{\D} |f (z)|^2 \, \omega (1 - |z|^2) \, dA (z) < \infty \, .
\end{equation} 

Such spaces are used in \cite{Karim-Pascal} and in \cite{LQR-radius}. We have $H_\omega^2 = H^2 (\beta)$ with:
\begin{equation} \label{coto}
\beta_n = 2 \int_{0}^{1} r^{2 n + 1} \omega (1 - r^2) \, dr = \int_{0}^{1} t^{n} \omega (1 - t ) \, dt \, .
\end{equation} 
Indeed, since $\beta_n = \int_{0}^{1}(1 -  t)^{n} \omega ( t ) \, dt$, the sequence $\beta = (\beta_n)_n$ is decreasing. Moreover, the fact that the 
automorphisms of $\D$ induce bounded composition operators on $H_\omega^2$ is proved as in  the Remark before Theorem~\ref{theo compos op}, namely:
\begin{align*}
\Vert f \circ \varphi_a \Vert_{H_\omega^2}^2 
& = \int_{\D} |f (w)|^2 \, |\varphi_{a}' (w)|^2 \, \omega (1 - |\varphi_{a} (w)|^2)  \, dA (w) \\
& \leq \bigg( \frac{1 + |a|}{1 - |a|} \bigg)^2 
\int_{\D} |f (w)|^2 \, \rho \bigg(\frac{1 - |\varphi_{a} (w)|^2}{1 - |w|^2} \bigg) \, \omega ( 1 - |w|^2)  \, dA (w) \\
& \leq \bigg( \frac{1 + |a|}{1 - |a|} \bigg)^2 c_{\rho, a} \int_{\D} |f (w)|^2 \, \omega(1-|w|^2) \, dA (w) \\
& =: \kappa_a \, \Vert f \Vert_{H_\omega^2}^2 \, ,
\end{align*}
where we used that $| \phi_a ' (w)| \leq \frac{1 + |a|}{1 - |a|}$, that $\frac{1 - |\varphi_{a} (w)|^2}{1 - |w|^2} \leq \frac{1 + |a|}{1 - |a|}$, and that 
$\rho$ is locally bounded. 
\smallskip

Note that  ${\mathfrak B}_{\gamma}^{2}$, $\gamma > - 1$, corresponds to $\omega (t) = (\gamma  + 1) \, t^\gamma$.

%
\begin{proof} [Proof of Proposition~\ref{prop admissible}] 
Condition $(i)$ is satisfied because $\beta$ is decreasing. Moreover, since $\beta$ is decreasing, Theorem~\ref{theo compos op-bis}, applied to 
$T = M_w$, with $w \in H^\infty$, ensures that 
$H^\infty = {\mathcal M} (H^2) \subseteq {\mathcal M} \big(H^2 (\beta) \big)$, and, for all $w \in H^\infty$:
\begin{displaymath} 
\| M_w \colon H^2 (\beta) \to H^2 (\beta) \| \leq \| M_w \colon H^2 \to H^2 \| = \| w \|_\infty \, .
\end{displaymath} 
Now, Proposition~\ref{multipliers} implies that $H^\infty = {\mathcal M} \big(H^2 (\beta) \big)$ and 
\begin{displaymath} 
\| M_w \colon H^2 (\beta) \to H^2 (\beta) \| = \| w \|_\infty 
\end{displaymath} 
for all  $w \in H^\infty$. 

It remains to show that, for $H = H^2 (\beta)$, the condition $(iii)$ implies the condition $(iv)$.

Since $H^2 (\beta)$ is isometrically rotation invariant, it is clear that $\|K_a\|=\|K_{|a|}\|$; hence, by the maximum principle, 
$\|K_x \| \le \|K_y \|$, for $0 \le x \le y < 1$. 

Assume now that $0 < y < x < 1$. Let $T$ be the disk automorphism:
\begin{equation} \label{automorphism}
T (z) = \frac{ 2 z + 1}{z + 2} \, , \qquad z \in \D \, .
\end{equation} 
The fixed points of $T$ are $1$ and $-1$, and $T (0) = 1/2$. We define the sequence $(a_n)_{n \ge 0}$ by induction, with:
\begin{displaymath} 
a_0 = 0 \, , \qquad a_{n + 1} = T (a_n) \, .
\end{displaymath} 
We see that:
\begin{equation} \label{primera}
1 - a_{n + 1} = \int_{a_n}^1 T ' (x) \, dx = \int_{a_n}^1 \frac{3}{(x + 2)^2} \, dx \le \frac{3}{4}\, (1 - a_n) \, ;
\end{equation}
so $(a_n)_n$ is increasing and converges to $1$. In the same way, we see that:
\begin{equation} \label{segunda}
1 - a_{n + 1} = \int_{a_n}^1 T ' (x) \, dx = \int_{a_n}^1 \frac{3}{(x + 2)^2} \, dx \ge \frac{1}{3} \,  (1 - a_n) \, .
\end{equation}

Since $0 < y < x < 1$, we can find $m \le n$ such that:
\begin{displaymath} 
a_{m - 1} < y < a_m \, , \qquad \text{and}\qquad a_{n - 1} < x < a_n \, .
\end{displaymath} 
We have $\| K_x \| \leq \| K_{a_n} \|$ and $\| K_y \| \geq \| K_{a_{m - 1}} \|$. 
Since $C_T^* K_z = K_{T (z)}$ for all $z \in \D$, we have:
\begin{displaymath} 
\frac{\|K_x \|}{\|K_y \|} \le \frac{\|K_{a_n}\|}{\|K_{a_{m - 1}}\|} \le \| C_T^*\|^{n - m + 1} = \alpha^{n - m + 1} \, ,
\end{displaymath} 
with $\alpha = \|C_T\| \ge 1$. 
Applying \eqref{primera} and \eqref{segunda}, we get:
\begin{displaymath} 
\frac{1 - y}{1 - x} \ge \frac{1 - a_m}{1 - a_{n - 1}} \ge \frac{1 - a_m}{3 (1 - a_{n})}
\ge \frac{1}{3} \Bigl( \frac{4}{3}\Bigr)^{n - m} \, .
\end{displaymath} 

It suffices now to take $s \ge 0$ such that $(4/3)^s = \alpha$, and $A > 0$ large enough in order that, with the increasing function 
$h (t) = \max\{A t^s, 1\}$, $t > 0$, we have:
\begin{displaymath} 
h \Bigl(\frac{1 - y}{1 - x}\Bigr) \ge \frac{A \,}{3^s} \, \alpha^{n - m} \ge \frac{A \,}{3^s \alpha} \frac{\|K_x\|}{\|K_y\|}
\ge \frac{\|K_x\|}{\|K_y \|}\, \cdot \qedhere
\end{displaymath} 
\end{proof}
\bigskip

Let us come back to the conditional multipliers. 
In general, we obviously have:
\begin{equation} \label{inclusion relative multipliers}
H^\infty \subseteq \mathcal{M} (H, \varphi) \subseteq H \, .
\end{equation} 

The extreme cases were characterized by Attele (\cite{Attele-2}) when $H = H^2 = {\mathfrak B}_{- 1}^{2}$ (and 
Contreras and Hern{\'a}ndez-D{\'\i}az in \cite{Contreras} for the spaces $H^p$)  as follows. 
\smallskip\goodbreak

\begin{theorem} [Attele] \label{th-attele}
We have: 
\begin{enumerate}
\setlength\itemsep {-0.05 em}

\item [$1)$] $\mathcal{M} (H^2, \varphi) = H^2$ if and only if $\Vert \varphi \Vert_\infty < 1$. 

\item [$2)$] $\mathcal{M} (H^2, \varphi) = H^\infty$ if and only if $\varphi$ is a finite Blaschke product.
\end{enumerate}
\end{theorem}

A key tool for the most delicate second necessary condition is the use of inner and outer functions. We no longer have this tool at our disposal for the 
admissible spaces $H = H^{2}(\beta)$, but we can nevertheless state the following analog result.
\goodbreak

\begin{theorem} \label{th-nopom}  
Let $\varphi$ an analytic self-map of $\D$ and $H$ be an admissible Hilbert space on which $C_\phi$ acts boundedly.  We have: 
\begin{enumerate}
\setlength\itemsep {-0.05 em}

\item [$1)$] $\mathcal{M} (H^2, \varphi) \subseteq \mathcal{M} (H,\varphi)$; 

 \item [$2)$] $\mathcal{M} (H,\varphi) = H$ if and only if $\Vert \varphi \Vert_\infty < 1$;
 
\item [$3)$] $\mathcal{M} (H, \varphi) = H^\infty$ if and only if $\varphi$ is a finite Blaschke product.
\end{enumerate}
\end{theorem}

Note that the assumption that $C_\phi$ acts boundedly on $H$ is automatically satisfied when $H = H^2 (\beta)$ with $\beta$ decreasing, by 
Theorem~\ref{theo compos op}. 
\goodbreak
\begin{proof} \par
$1)$ Suppose first that $\varphi (0) = 0$. Let $w \in \mathcal{M} (H^2, \varphi)$. The weighted composition operator $M_w C_\varphi$ is bounded on $H^2$, 
and moreover lower triangular on the canonical basis; applying Theorem~\ref{theo compos op-bis},~1), we get that $M_w C_\varphi$ is bounded on $H$ as well,  
that is $w \in \mathcal{M} (H,\varphi)$. 

In the general case, let $\varphi (0) = a$, so that $(\varphi_a \circ \varphi) (0) = 0$.  Property $(iii)$  implies that 
\begin{displaymath} 
\mathcal{M} (H^2,\varphi) = \mathcal{M} (H^2, \varphi_a \circ \varphi) \subseteq \mathcal{M} (H,\varphi_a \circ \varphi) = \mathcal{M} (H,\varphi)  \, ,
\end{displaymath} 
since $f \in H^2$ if and only if $f \circ \varphi_a\in H^2$ and $f \in H$ if and only if $f \circ \varphi_a \in H$. 
\smallskip

$2)$ The necessary condition is proved as in \cite{Attele-2} for $H^2$; we recall some details. We start from the (obvious, but useful) mapping equation:
\begin{equation} \label{maeq} 
(M_w C_\varphi)^{\ast} (K_z) = \overline{w (z)} \, K_{\varphi (z)} \, .
\end{equation}
The assumption implies the existence of a constant $C$ such that:
\begin{displaymath} 
\qquad \qquad \qquad \qquad \Vert M_w C_\varphi \Vert _{\mathcal{L}(H)}\leq C \, \Vert w \Vert_H \qquad \text{ for all } w \in H \, .
\end{displaymath} 
As a consequence, for given $z \in \D$:
\begin{displaymath} 
\Vert (M_{w} C_{\varphi})^{\ast} (K_z) \Vert_{H} \leq C \, \Vert w \Vert_H \Vert K_z \Vert_H \, ,
\end{displaymath} 
that is, in view of \eqref{maeq}:
\begin{equation} \label{maeqbis} 
|w (z)| \, \Vert K_{\varphi (z)} \Vert_{H} \leq C \, \Vert w \Vert_H \Vert K_z \Vert_H \, .
\end{equation} 
Testing this inequality with $w = K_z$ and simplifying by $\Vert K_z \Vert_{H}^2$, we get that $\Vert K_{\varphi (z)} \Vert_{H}\leq C$. 
Since $\lim_{|a|\to 1^-} \Vert K_a \Vert_H = \infty$, as a consequence of $(i)$, this implies that $\Vert \varphi \Vert_\infty  < 1$, by this same consequence. 
\smallskip

For the sufficient condition, observe that if $\| \phi \|_\infty < 1$, then $f \circ \phi \in H^\infty$ for all $f \in H$.  Since $\mathcal{M} (H) = H^\infty$, 
according to $(iv)$, we get that $f \circ \phi \in \mathcal{M} (H)$ and therefore $w (f \circ \phi) = M_{f \circ \phi} \, w \in H$ for all $w \in H$. That means 
that $H \subseteq {\mathcal M} (H, \phi)$. Therefore, by \eqref{inclusion relative multipliers}, we have ${\mathcal M} (H, \phi) = H$.
\smallskip

$3)$ The sufficient condition goes as follows: finite Blaschke products $\varphi$ clearly satisfy (and actually are characterized by: see \cite{Attele-2}):
\begin{displaymath} 
R_\varphi := \sup_{z \in \D} \frac{1 - |\varphi (z)|}{1 - |z|} < \infty \, . 
\end{displaymath} 
Let now $w \in \mathcal{M} (H, \varphi)$, so that $C := \Vert M_w C_\varphi \Vert < \infty$. We may assume that $\Vert w \Vert_H \leq 1$. The mapping 
equation \eqref{maeqbis} gives, for $z \in \D$:
\begin{displaymath} 
|w (z)| \, \Vert K_{\varphi (z)} \Vert_H \leq C \, \Vert K_{z} \Vert_H \, .
\end{displaymath} 
By $(ii)$, this implies that, for $|z|$ close enough to $1$:
\begin{displaymath} 
|w (z) | \leq C \, \frac{\Vert K_z \Vert_H}{\Vert K_{\varphi (z)} \Vert_H} \leq C \, h\bigg( \frac{1 - |\varphi (z)|}{1 - |z|} \bigg) 
\leq C \, h (R_\varphi) \, .
\end{displaymath} 
This means that $w \in H^\infty$. 
\smallskip

Finally, for the  necessary condition, assume that $\mathcal{M} (H, \varphi) = H^\infty$. Then $\mathcal{M} (H^2, \varphi) = H^\infty$, by $1)$, and then 
$\varphi$ is a finite Blaschke product by Attele's theorem (Theorem~\ref{th-attele}). 
\end{proof}
\smallskip

\noindent{\bf Remark.} In Proposition~\ref{prop admissible}, we assume that the automorphisms of $\D$ induce bounded composition operators on 
$H^2 (\beta)$. It is known (\cite[Theorem~1]{KRMA}) that this is not always the case. Let us give a simpler proof, for a particular case. 
Let $\beta_n = \exp (- \sqrt n)$, and consider the space $H^2 (\beta)$. We then have, for $0 < r = \e^{- \eps} < 1$:
\begin{displaymath} 
\Vert K_r \Vert^2 = \sum_{n = 0}^\infty r^{2 n} \exp (\sqrt n) = \sum_{n = 0}^\infty \e^{- 2 n \varepsilon} \exp (\sqrt n) =: S (\varepsilon) \, .
\end{displaymath} 
We easily see (using e.g. the Euler-MacLaurin formula) that, when $\varepsilon \to 0^{+}$:  
\begin{displaymath}  
S (\varepsilon) \sim I (\varepsilon) := \int_{0}^\infty \exp (\sqrt{t} - 2 \, \varepsilon t) \, dt
= (4 \varepsilon^{2})^{- 1} \int_{0}^\infty \exp \bigg( \frac{\sqrt{x} - x}{2 \, \varepsilon} \bigg) \, dx \, .
\end{displaymath} 
We use the Laplace theorem (\cite{DIE}   p.125) on the equivalence of integrals:
\begin{displaymath} 
\int_{0}^\infty \e^{A \varphi (x)} \, dx \sim \sqrt{2 \pi(|\varphi '' (x_0)|)^{- 1}}\, A^{- 1/2} \e^{A \varphi (x_0)} \, , \quad \text{as } A \to \infty \, ,
\end{displaymath} 
and apply it to $A = 1 / 2\varepsilon$ and to the function $\varphi (x) = \sqrt x - x$, which takes its maximum at $x_0 = 1/4$, with $\varphi (x_0) = 1/4$. 
We get that:
\begin{displaymath} 
S (\varepsilon) \approx \varepsilon^{- 3/2} \exp \bigg( \frac{1}{8\varepsilon}\bigg) \approx (1 - r)^{- 3/2} \exp\bigg( \frac{1}{8 (1 - r)} \bigg) \, .
\end{displaymath}
Now, consider the automorphism $T$ of $\D$ given by \eqref{automorphism}. For $r < 1$, we have 
$1 - T (r) \sim (1 - r) T ' (1) = (1 - r) / 3$; so:
\begin{displaymath} 
\frac{1}{1 - T (r)} - \frac{1}{1 - r} = \frac{1}{1 - r} \bigg( \frac{1 - r}{1 - T (r)} - 1 \bigg) \sim  \frac{1}{1 - r} \bigg( \frac{1}{T ' (1)} - 1\bigg) 
=  \frac{2}{1 - r} \, \raise 1 pt \hbox{,}
\end{displaymath} 
and that implies that:
\begin{displaymath}  
\frac{\Vert K_{T (r)} \Vert^2}{\Vert K_r\Vert^2} \approx 
\exp \bigg[ \frac{1}{8}\bigg(\frac{1}{1 - T (r)} - \frac{1}{1 - r} \bigg) \bigg] 
\converge_{r \to 1^-} \infty \, .
\end{displaymath}
Since $K_{T (r)} = C_T^\ast (K_r)$, this implies that $C_T^\ast$, and hence $C_T$, is not bounded on  $H^{2}(\beta)$.
\goodbreak

\section{\!\! Schatten classes for Hardy spaces and Bergman spaces} \label{Hardy-Bergman} 

We know that if a composition operator $C_\phi$ is compact on the Hardy space $H^2$, then it is compact on the Bergman space ${\mathfrak B}^2$ 
(see \cite[Proposition~2.7 and Theorem~3.5]{MacCluer-Shapiro}. Theorem~\ref{theo pas Bergman-Schatten} below shows that we cannot expect better. 
\smallskip

Let us begin by a preliminary result. Recall that the $2$-Carleson function of the analytic map $\phi \colon \D \to \D$ is:
\begin{equation} 
\rho_{\phi, 2} (h) = \sup_{\xi \in \T} A_\phi \big( W (\xi, h) \big) \, , 
\end{equation} 
where $A$ is the normalized area measure on $\D$, $A_\phi$ is the pull-back measure of $A$ by $\phi$, i.e. $A_\phi (B) = A [\phi^{- 1} (B)]$ for all Borel 
sets $B \subseteq \D$. 
It is well-known (see \cite{Hastings}) that $\rho_{\phi, 2} (h) = {\rm O}\, (h^2)$, due to the fact that all 
composition operators $C_\phi$ are bounded on ${\mathfrak B}^2$, and that $C_\phi$ is compact on ${\mathfrak B}^2$ if and only if 
$\rho_{\phi, 2} (h) = {\rm o}\, (h^2)$. For Schatten classes, we have the following result. 

\begin{proposition} \label{lemme Schattenitude} 
If the composition operator $C_\phi \colon {\mathfrak B}^2 \to {\mathfrak B}^2$ is in the Schatten class $S_p ({\mathfrak B}^2)$ for some $p \in (0, \infty)$, 
then:
\begin{equation} \label{formule Schattenitude} 
\rho_{\phi, 2} (h) = {\rm o}\, \bigg( h^2 \Big( \log \frac{1}{h} \Big)^{- 2 / p} \bigg) \, .
\end{equation} 
\end{proposition}
\begin{proof}
We follow the proof of \cite[Proposition~3.4]{LLQR-2008}. By \cite[Corollary~2]{Luecking} and \cite[Proposition~3.3]{LLQR-2008}, 
$C_\phi \in S_p ({\mathfrak B}^2)$ if and only if:
\begin{equation} \label{condition Luecking}
\sum_{n = 1}^\infty \bigg( \sum_{j = 0}^{2^n - 1} 4^{n p / 2} [ A_\phi (W_{n, j}) ]^{p / 2}\bigg)  < \infty \, ,
\end{equation} 
where $W_{n, j} = W (\e^{2 j i \pi / 2^n}, 2^{- n})$.

Observing that, for $h = 2^{- n}$, we have:
\begin{displaymath}
\big[\rho_{\phi, 2} (2^{- n}) \big]^{p/2} \leq \sum_{j = 0}^{2^n - 1} \big[ A_\phi (W_{n, j}) \big]^{p/2}  \, ,
\end{displaymath}
\eqref{condition Luecking} yields:
\begin{displaymath}
\sum_{n = 1}^\infty \big[ \rho_{\phi, 2} (2^{- n})\big]^{p/2} 4^{n p/2} < +\infty \, .
\end{displaymath}

By \cite[Theorem~3.1]{LLQR-TAMS}, we have a constant $C_0 > 0$ such that:
\begin{displaymath} 
\rho_{\phi, 2} (\eps h) \leq C_0 \, \eps^2 \rho_{\phi, 2} (h) 
\end{displaymath} 
for $0 < \eps \leq 1$ and $0 < h < 1$. Hence, if we set:
\begin{displaymath} 
u_n = \bigg( \frac{\rho_{\phi, 2} (2^{- n})}{4^{- n}} \bigg)^{p / 2} \, \raise 1 pt \hbox{,}
\end{displaymath} 
we have, for $n \geq k$:
\begin{displaymath} 
u_n \leq C_0^{p / 2} \, u_k \, . 
\end{displaymath} 
The following lemma, whose proof is postponed, then shows that:
\begin{equation}\label{eq condition necessaire}
n \, \bigg(\frac{\rho_{\phi, 2} (2^{- n})}{4^{- n}}\bigg)^{p/2} \converge_{n \to \infty} 0 \, .
\end{equation}
\begin{lemma} \label{classical lemma} 
Let $\sum u_n$ be a convergent series of positive numbers such that $u_n \leq C \, u_k$ for $n \geq k$, for some positive constant $C$. Then 
$u_n =  {\rm o}\, (1 / n)$. 
\end{lemma}

To finish the proof, it remains to consider, for every $h \in (0,1/2)$, the integer $n$ such that 
$2^{- n - 1} < h \leq 2^{- n}$; then \eqref{eq condition necessaire} gives:
\begin{displaymath}
\lim_{h \to 0^+} \bigg( \frac{\rho_{\phi, 2} (h)}{h^2} \bigg)^{p/2} \log(1/h) = 0 \, ,
\end{displaymath}
as announced. 
\end{proof}
\begin{proof} [Proof of Lemma~\ref{classical lemma}] 
It is classical. We recall it for convenience. Let:
\begin{displaymath} 
v_n = \sum_{n / 2 < k \leq n} u_k \, .
\end{displaymath} 
Since the series $\sum u_n$ converges, on the one hand, we have $v_n \converge_{n \to \infty} 0$, and on the other hand: 
\begin{displaymath} 
v_n \geq \frac{n}{2} \, C^{- 1} u_n \, . \qedhere
\end{displaymath} 
\end{proof}
\goodbreak

\begin{theorem} \label{theo pas Bergman-Schatten} 
There exists a symbol $\phi$ for which the composition operator $C_\phi$ is compact on the Hardy space $H^2$, but is not in any Schatten class 
$S_p ({\mathfrak B}^2)$ of the Bergman space ${\mathfrak B}^2$ with $p < \infty$.
\end{theorem} 
\begin{proof}
We use a variant of the Shapiro-Taylor map (\cite[Section~4]{Shapiro-Taylor}) introduced in \cite[Theorem~5.6]{LLQR-2008} for showing that there is 
a compact composition operator on $H^2$ which is in no Schatten class $S_p (H^2)$ for $p < \infty$. Let:
\begin{equation} \label{V eps}
V_\eps = \{ z \in \C \tq \Re z > 0 \text{ and } |z| < \eps \} 
\end{equation} 
We set:
\begin{equation} 
f (z) = z \log ( - \log z) \, ,
\end{equation} 
where $\log$ is the principal determination of the logarithm. For $\eps > 0$ small enough, we have $\Re f (z) > 0$  for $z \in V_\eps$. Let 
$g \colon \D \to V_\eps$ be the conformal map from $\D$ onto $V_\eps$ sending $\T = \partial \D$ onto $\partial V_\eps$, and with $g (1) = 0$ 
and $g' (1) = - \eps / 4$. 
Explicitly, $g$ is the composition of the following maps: a) $\sigma \colon z \mapsto -z$ from $\D$ onto itself; 
b) $\gamma \colon z \mapsto \frac{z + i}{1 + iz}$ from $\D$ onto $P=\{ \Im z > 0\}$; 
c) $s \colon z \mapsto \sqrt z$ from $P$ onto $Q =\{\Re z > 0\,, \Im z >0\}$; 
d) $\gamma^{-1} \colon z \mapsto \frac{ z -i }{ 1 -iz}$ from $Q$ onto $V = \{ |z | < 1\,, \Re z >0\}$, and 
e) $h_\eps \colon z \mapsto \eps z$ from $V$ onto $V_\eps$.

We then set:
\begin{equation} \label{Shapiro-Taylor map}
\phi = \exp ( - f \circ g) \, .
\end{equation} 

This analytic function $\phi$ maps $\D$ into itself and we proved in \cite[Theorem~5.6]{LLQR-2008} that $C_\phi$ is compact on $H^2$. 

For $z = r \e^{i \alpha} \in V_\eps$, we have (see \cite[proof of Theorem~5.6]{LLQR-2008}):
\begin{align} 
\Re f (z) & = r \log \log \frac{1}{r} \bigg( \cos \alpha + \frac{\alpha \sin \alpha}{\log (1 / r) \log \log (1/r)} \bigg) \\
& \phantom{\Re f (z)  = r \log \log \frac{1}{r} \bigg( \cos \alpha + } + {\rm o}\, \bigg( \frac{1}{\log (1 /r) \log \log (1/r)} \bigg) \notag \\
\Im f (z) & = r \log \log \frac{1}{r} \bigg( \sin \alpha - \frac{\alpha \cos \alpha}{\log (1 / r) \log \log (1/r)} \bigg) \\
& \phantom{\Re f (z)  = r \log \log \frac{1}{r} \bigg( \cos \alpha + } + {\rm o}\, \bigg( \frac{1}{\log (1 /r) \log \log (1/r)} \bigg) \, . \notag
\end{align} 
It follows that:
\begin{equation}
0  < \Re f (z) \lesssim r \log \log \frac{1}{r}  \qquad \text{and} \qquad |\Im f (z)| \lesssim  r \log \log \frac{1}{r} \, \cdot
\end{equation}

Now, assume that $r \leq h / \log \log (1 / h)$. Then $r \log \log (1 / r) \lesssim h$. Since $g' (1) \neq 0$,  $g$ is bi-Lipschitz in a neighborhood of $1$; hence 
$|\phi (u) | \approx 1 - \Re f [g (u)]$ and $|\arg \phi (u)| \approx |\Im f [g (u)] |$, so we have:
\begin{displaymath} 
A_\phi \big( W (1, h) \big) \gtrsim \big( h / \log \log (1 / h) \big)^2. 
\end{displaymath} 
Therefore:
\begin{equation} 
\rho_{\phi, 2} (h) \geq A_\phi \big( W (1, h) \big) \gtrsim \frac{h^2}{ \big(\log \log (1 / h) \big)^2}  \, \raise 1 pt \hbox{,}
\end{equation} 
so \eqref{formule Schattenitude} cannot be satisfied. Hence $C_\phi \notin S_p ({\mathfrak B}^2)$, whatever $p < \infty$.
\end{proof}
\medskip\goodbreak

When $C_\phi \in S_p (H^2)$, we actually have a better behavior on the Bergman space.
\goodbreak

\begin{theorem} \label{Sp/2 Bergman}
For every $p > 0$, we have $C_\phi \in S_{p / 2} ({\mathfrak B}^2)$ when $C_\phi \in S_p (H^2)$.
\end{theorem} 

In particular, if $C_\phi$ is Hilbert-Schmidt on $H^2$, then it is nuclear on ${\mathfrak B}^2$ (since on Hilbert spaces the nuclear operators coincide with those 
in the Schatten class $S_1$). 

\begin{proof}
We proved in \cite[formula (3.26), page 3963]{LLQR-TAMS}, as a consequence of the main result of \cite{LLQR-Math-Ann}, that for some positive constants 
$C, C'$, we have:
\begin{displaymath} 
A_\phi [W (\xi, h)] \leq C \, \big( m_\phi [W (\xi, C h) ] \big)^2 
\end{displaymath} 
for all $\xi \in \T$ and $0 < h < 1$ small enough. We may assume, enlarging $C'$ if needed, that $C' = 2^N$ for some positive integer $N$.
Hence if $W_{n, j} = W (\e^{2 j i \pi / 2^n}, 2^{- n})$ and 
$W'_{n, j} = W (\e^{2 j i \pi / 2^n}, 2^N 2^{- n})$, for $n > N$, we have:
\begin{displaymath} 
\sum_{j = 0}^{2^n - 1} \big( 4^n A_\phi ( W_{n, j}) \big)^{p / 4} \leq C \sum_{j = 0}^{2^n - 1} \big( 2^n m_\phi (W'_{n, j}) \big)^{p / 2} \, .
\end{displaymath} 
Now, each Carleson window $W'_{n, j}$ of size $2^N 2^{- n}$ is contained in the union of $2^N$ other Carleson windows 
$W_{n, j_1}, \ldots, W_{n, j_{2^N}}$ of size $2^{- n}$ and of less than $N 2^{N  - 1}$ Hastings-Luecking boxes $R_{\nu, j_l}$ with $\nu \leq n - 1$. Hence:
\begin{displaymath} 
\begin{split}
\sum_{j = 0}^{2^n - 1} \big( 4^n A_\phi ( W_{n, j}) \big)^{p / 4} 
\leq C_p 2^N \sum_{j = 0}^{2^n - 1} & \big( 2^n m_\phi (W_{n, j}) \big)^{p / 2} \\ 
& + C_p N 2^{N - 1} \sum_{j = 0}^{2^n - 1} \big( 2^n m_\phi (R_{n, j}) \big)^{p / 2} \, ,
\end{split}
\end{displaymath} 
(since, for $a, b \geq 0$, we have $(a + b)^r \leq a^r + b^r$ if $0 < r \leq 1$, and $(a + b)^r \leq 2^{r - 1} (a^r + b^r)$ if $r \geq 1$). 
\smallskip

It follows, thanks to \cite[Corollary~2]{Luecking} and \cite[Proposition~3.3]{LLQR-2008}, that $C_\phi \in S_p (H^2)$ implies 
$C_\phi \in S_{p / 2} ({\mathfrak B}^2)$. 
\end{proof}
\medskip

Theorem~\ref{Sp/2 Bergman} is sharp, as said by the following result.
\goodbreak

\begin{theorem} \label{Sp/2 pas mieux}
For every $p$ with $0 < p < \infty$, there exists a symbol $\phi$ for which the composition operator $C_\phi$ is in the Schatten class $S_p (H^2)$ on the Hardy 
space, but is not in any Schatten class $S_q ({\mathfrak B}^2)$ of the Bergman space with $q < p / 2$. 
\end{theorem} 

Before giving the proof, let us mention that this theorem implies (in a strong way) a separation between Schatten classes by composition operators on Bergman 
spaces. Curiously, we did not find any reference for this result.

Indeed, for every $r > 0$, there exists a symbol $\phi$ for which the composition operator $C_\phi$ is in the Schatten class $S_{2r} (H^2)$ on the Hardy space, 
hence in the Schatten class $S_r({\mathfrak B}^2)$ on the Bergman space by Theorem~\ref{Sp/2 Bergman}, but which is not in any Schatten class 
$S_q ({\mathfrak B}^2)$ of the Bergman space for $q < r$.

\begin{proof}
Again, we use the variant of the Shapiro-Taylor map introduced in \cite[Theorem~5.4]{LLQR-2008} in order to have a composition operator in $S_p (H^2)$ but 
not in $S_q (H^2)$ for $q < p\,$. For $\eps > 0$ small enough, we set, for $z \in V_\eps$, where $V_\eps$ is defined in \eqref{V eps}:
\begin{equation} 
f (z) = z (- \log z)^{2 / p} [\log (- \log z)]^s \, ,
\end{equation} 
with $s > 1 / p\,$.

We set:
\begin{equation} \label{modified Shapiro-Taylor map}
\phi = \exp (- f \circ g) \, ,
\end{equation} 
where $g$ is as in the proof of Theorem~\ref{theo pas Bergman-Schatten}. Then $\phi \colon \D \to \D$ is analytic and we proved in 
\cite[Theorem~5.4]{LLQR-2008} that $C_\phi \in S_p (H^2)$. 

For $z = r \e ^{i \alpha} \in V_\eps$, we have (see \cite[Lemma~5.5]{LLQR-2008}:
\begin{align} 
0 < \Re f (z) & \lesssim r \bigg( \log \frac{1}{r} \bigg)^{2 / p} \bigg( \log \log \frac{1}{r} \bigg)^s \\
|\Im f (z) | & \lesssim r \bigg( \log \frac{1}{r} \bigg)^{2 / p} \bigg( \log \log \frac{1}{r} \bigg)^s \, .
\end{align} 

As in the proof of Theorem~\ref{theo pas Bergman-Schatten}, that implies that:
\begin{equation} \label{cond Schatten-Bergman} 
\rho_{\phi, 2} (h) \geq A_\phi \big( W (1, h) \big) \gtrsim \frac{h^2}{(\log 1/ h)^{4 / p} (\log \log 1 / h)^{2 s}} 
\end{equation} 
By Proposition~\ref{lemme Schattenitude}, if $C_\phi$ is in $S_q ({\mathfrak B}^2)$, we have:
\begin{displaymath} 
\rho_{\phi, 2} (h) = {\rm o} \bigg( \frac{h^2}{(\log 1 / h)^{2 / q}} \bigg) \, \raise 1 pt \hbox{,}
\end{displaymath} 
but, due to \eqref{cond Schatten-Bergman}, this is not possible for $q < p / 2$. Therefore $C_\phi \notin S_q$.
\end{proof}

\noindent {\bf Remark.} Actually Theorem~\ref{Sp/2 Bergman} has a more general form.
\goodbreak

\begin{theorem} \label{generalization}
Let ${\mathfrak B}_{\gamma_1}^2$ and ${\mathfrak B}_{\gamma_2}^2$ be two weighted Bergman space of parameter 
$\gamma_1$ and $\gamma_2$, with $\gamma_2 > \gamma_1 \geq - 1$. Then, for every $p > 0$ and any symbol $\phi$, we have:
\begin{itemize}
\setlength\itemsep {-0.1 em}
\item [$1)$] $C_\phi \in S_p ({\mathfrak B}_{\gamma_1}^2)$ implies $C_\phi \in S_{\tilde p} ({\mathfrak B}_{\gamma_2}^2)$, with 
$\tilde p = \frac{\gamma_1 + 2}{\gamma_2 + 2} \, p < p$; 

\item [$2)$] when $\phi$ is finitely valent, the converse holds. 

\end{itemize}
\end{theorem} 

Note that this theorem gives another, though less explicit, proof of Theorem~\ref{theo pas Bergman-Schatten} and of Theorem~\ref{Sp/2 pas mieux}, as a 
direct consequence of \cite[Theorem~5.4 and Theorem~5.6]{LLQR-2008} since the symbols used in the proof of these theorems (and in that of the above 
Theorem~\ref{theo pas Bergman-Schatten} and Theorem~\ref{Sp/2 pas mieux}) are univalent. In fact, that the Shapiro-Taylor map, defined in 
\eqref{Shapiro-Taylor map}, is univalent is proved in \cite[Lemma~4.1~(a)]{Shapiro-Taylor}. As well,  the modified Shapiro-Taylor map, defined in 
\eqref{modified Shapiro-Taylor map}, is univalent. In fact, its derivative $f ' (z)$ is the sum of three terms, the dominant one being 
$(- \log z)^{2/p}\,\big( \log( - \log z) \big)^s$; it ensues that, for $\eps$ small enough and $z = r \, \e^{i t}$ with $0 < r < \eps$ and $|t| < \pi/2$, 
we have:
\begin{displaymath} 
\Re  f ' (z) \geq \frac{1}{2} (\log 1/r)^{2/p} (\log\log 1/r)^s > 0 \, ,
\end{displaymath} 
and it follows that $f$ is univalent in $V_\eps$. In both cases, the symbol $\phi = \exp (- f \circ g)$ is univalent. 

\begin{proof}
Recall that D.~Luecking and K.~H.~Zhu proved in \cite[Theorem~1 and Theorem~3]{Luecking-Zhu} that, for $\gamma \geq - 1$, we have 
$C_\phi \in S_p ({\mathfrak B}_{\gamma}^2)$ if and only if:
\begin{equation} 
\int_\D \bigg( N_{\phi, \gamma + 2} (z) \, \Big( \log \frac{1}{|z|} \Big)^{- (\gamma + 2)} \bigg)^{p / 2} \, \frac{dA (z)}{(1 - |z|^2)^2} < \infty \, , 
\end{equation} 
where $N_{\phi, \beta}$ ($\beta \geq 1$) is the \emph{weighted Nevanlinna counting function}, defined as:
\begin{equation} \label{weighted Nevanlinna} 
N_{\phi, \beta} (z) = \sum_{\phi (w) = z} \Big( \log \frac{1}{|w|} \Big)^\beta 
\end{equation} 
if $z \in \phi (\D) \setminus \{\phi (0) \}$, and $N_{\phi, \beta} (z) = 0$ otherwise. 

As said in the introduction, for $\gamma = -1$ we have ${\mathfrak B}_{- 1}^2 = H^2$. 

Now, for $1 \leq \beta_1 < \beta_2$, the $\ell_{\beta_2}$-norm is smaller than the $\ell_{\beta_1}$-norm; so we have:
\begin{equation}
\big[ N_{\phi, \beta_2} \big]^{1 / \beta_2} \leq \big[ N_{\phi, \beta_1} \big]^{1 / \beta_1} \, .
\end{equation} 
It follows that, for $- 1 \leq \gamma_1 < \gamma_2$:
\begin{align*}
\int_\D \bigg( N_{\phi, \gamma_2 + 2} (z) \, & \Big( \log \frac{1}{|z|} \Big)^{- (\gamma_2 + 2)} \bigg)^{\tilde p / 2} \, \frac{dA (z)}{(1 - |z|^2)^2} \\ 
& \leq \int_\D \bigg( \big[ N_{\phi, \gamma_1 + 2} (z) \big]^{\frac{\gamma_2 + 2}{\gamma_1 + 2}} 
\, \Big( \log \frac{1}{|z|} \Big)^{- (\gamma_2 + 2)} \bigg)^{\tilde p / 2} \, \frac{dA (z)}{(1 - |z|^2)^2} \\
& = \int_\D \bigg( N_{\phi, \gamma_1 + 2} (z) \, \Big( \log \frac{1}{|z|} \Big)^{- (\gamma_1 + 2)} \bigg)^{p / 2} \, \frac{dA (z)}{(1 - |z|^2)^2} 
\, \raise 1 pt \hbox{,}
\end{align*}
and that proves that $C_\phi \in S_{\tilde p}  ({\mathfrak B}_{\gamma_2}^2)$ if $C_\phi \in S_p ( {\mathfrak B}_{\gamma_1}^2)$.
\smallskip

Now, if $\phi$ is $s$-valent, we have:
\begin{displaymath} 
n_\phi (z) := \sum_{\phi (w) = z} 1 = {\rm card}\, \{w \in \D \tq \phi (w) = z\} \leq s \, .
\end{displaymath} 
Using H\"older's inequality, we get, for $1 \leq \beta_1 < \beta_2$:
\begin{displaymath} 
N_{\phi, \beta_1} (z) \leq [n_\phi (z)]^{(\beta_2 - \beta_1) / \beta_2} \big[ N_{\phi, \beta_2} (z) \big]^{\beta_1 / \beta_2} 
\leq s^{(\beta_2 - \beta_1) / \beta_2} \big[ N_{\phi, \beta_2} (z) \big]^{\beta_1 / \beta_2} \, .
\end{displaymath} 
Therefore:
\begin{align*}
\int_\D & \bigg( N_{\phi, \gamma_1 + 2} (z) \, \Big( \log \frac{1}{|z|} \Big)^{- (\gamma_1 + 2)} \bigg)^{p / 2} \, \frac{dA (z)}{(1 - |z|^2)^2} \\
& \leq s^{ p \, (\beta_2 - \beta_1) / 2 \beta_2} \int_\D \bigg( N_{\phi, \gamma_2 + 2} (z) 
\, \Big( \log \frac{1}{|z|} \Big)^{- (\gamma_2 + 2)} \bigg)^{\tilde p / 2} \, \frac{dA (z)}{(1 - |z|^2)^2} \, \raise 1 pt \hbox{,}
\end{align*}
and $C_\phi \in S_{\tilde p} (({\mathfrak B}_{\gamma_2}^2)$ implies that $C_\phi \in S_p ( {\mathfrak B}_{\gamma_1}^2)$.
\end{proof}
\goodbreak

\section{Two examples} \label{examples}

\subsection {Preliminaries} 

Theorem \ref{theo compos op} can be successfully applied to $H^2 (\beta) = \mathcal{D}^{\, 2}_\alpha$ and  
$H^2 (\gamma) = \mathcal{D}^{\, 2}_{\alpha'}$ with $- 1 < \alpha < \alpha'$. But as we will see now with the example of the cusp map $\chi$, or of 
the lens maps, it does not provide as sharp estimates as wished. For example, with 
$H^2 (\beta) =\mathcal{D}^{\, 2} \subseteq H^2 (\gamma) = \mathcal{D}^{\, 2}_{\alpha}$ where $\alpha > 0$, it  gives, using 
\cite[Theorem~3.1]{LLQR-Arkiv}:
\begin{displaymath} 
a_{n} \big(C_\chi^{{\mathcal D}^{\, 2}_\alpha} \big) \lesssim \sqrt{ a_{n/2} (C_\chi^{{\mathcal D}^{\, 2}})} \lesssim \exp(- b \, \sqrt n) \, ,
\end{displaymath} 
while we will see, using the point of view of weighted composition operators, that actually:
\begin{displaymath} 
a_n (C_\chi^{{\mathcal D}^{\, 2}_\alpha}) \lesssim \exp \big(- b \, (n/\log n) \big) \, .
\end{displaymath} 

We now elaborate on this point of view.
\smallskip

Let $H$ be a Hilbert space of analytic functions on $\D$ whose set of multipliers $\mathcal{M} (H)$ is isometrically $H^\infty$. For example, this is the 
case for $H = {\mathfrak B}^2_\gamma$ for all $\gamma > - 1$, as recalled by Proposition~\ref{multipliers}.

Through a standard averaging argument, we easily have the following result (see \cite[Lemma~2.2]{LQR-JAT}).
\begin{proposition} \label{class} 
Let $H$ be a Hilbert space of analytic functions on $\D$ such that $\mathcal{M} (H) = H^\infty$. 
Let $z = (z_j)$ be a sequence of distinct points of $\D$ which is an interpolation sequence for $H^\infty$ with constant $I_z$. Then, the sequence 
$(K_{z_j})$ is a Riesz sequence for $H$ and moreover, for all $\lambda_1, \ldots, \lambda_n, \ldots \in \C$, we have:
\begin{displaymath} 
I_z^{- 2} \sum_j |\lambda_j|^2\,\| K_{z_j} \|^2 \leq \Big\| \sum_j \lambda_j K_{z_j} \Big\|^2 
\leq I_z^2 \sum_j |\lambda_j|^2 \, \| K_{z_j} \|^2 \, .
\end{displaymath} 
\end{proposition}

In \cite[Lemma 2.6]{GDHL}, we used Proposition~\ref{class} to prove an estimate from below (the proof was given only for $H = H^2$ and $w \in H^\infty$). 
We  slightly improve this estimate here,  with nearly the same proof, as follows. 

\begin{theorem} \label{sale} 
Let $H$ be a Hilbert space of analytic functions on $\D$ such that $\mathcal{M} (H) = H^\infty$. 
Let $\phi \colon \D \to \D$ be a symbol and $M_w \, C_\phi \colon H \to H$ an associated weighted composition operator with weight $w \in H$. 
We assume that $M_w \, C_\phi$ is bounded.  Let $u = (u_j)_{1 \leq j \leq n}$ be a sequence of length $n$ of points of $\D$ and $v_j = \phi (u_j)$, and 
assume that the points $v_j$ are distinct. Let $I_v$ be the interpolation constant of $v = (v_j)_{1 \leq j \leq n}$. Then, the approximation numbers of 
$M_w \, C_\phi$ satisfy:
\begin{displaymath} 
a_n (M_w \, C_\phi) \ge \inf_{1 \leq j \leq n} \bigg( |w (u_j)|\, \frac{\| K_{v_j} \|}{ \| K_{u_j} \|} \bigg) \, I_{v}^{- 2} \, .
\end{displaymath} 
\end{theorem}
%
  

%
\begin{proof}
Recall that the approximation numbers $a_{n} (S)$ of an operator $S$ on a Hilbert space coincide with its Bernstein numbers $b_{n} (S)$. Let $E$ be the span of 
$K_{u_j}$, with $1 \leq j \leq n$. Set $\delta = \inf_{1 \leq j \leq n} \frac{\Vert K_{v_j} \Vert}{\Vert K_{u_j} \Vert} \, |w (u_j)|$. Take 
$f = \sum_{j = 1}^n \lambda_j \, K_{u_j}$ in the unit sphere of $E$; we hence have 
$\sum_{j = 1}^n |\lambda_j|^2 \Vert K_{u_j} \Vert^2 \geq I_{u}^{- 2}$.  
Setting $T = M_w \, C_\phi$, we see that $T^{\ast} (f) = \sum_{j = 1}^n \lambda_j \overline{w (u_j)} \, K_{v_j}$, so that
\begin{align*}
\Vert T^{\ast} (f) \Vert^2 
& \geq  I_{v}^{- 2} \sum_{j = 1}^n |\lambda_j|^2 |w (u_j)|^2 \Vert K_{v_j}\Vert^2 
\geq  I_{v}^{- 2} \delta^2 \sum_{j = 1}^n |\lambda_j|^2 \Vert K_{u_j} \Vert^2 \\
& \geq \delta^2 I_{v}^{- 2} I_{u}^{- 2}  \geq \delta^2 I_{v}^{- 4} \, .
\end{align*}
In the last inequality, we used the obvious inequality $I_u \leq I_v$ (if $f (v_j) = a_j$, $j = 1, \ldots, n$, then $(f \circ \phi) (u_j) = a_j$ for $j = 1, \ldots, n$, and 
$\|f \circ \phi \|_\infty \leq \| f\|_\infty$). 
\smallskip

Hence $a_{n} (T) = a_{n} (T^{\ast}) = b_{n} (T^{\ast}) \geq \delta I_{v}^{- 2}$.
\end{proof}
\medskip
\goodbreak

In order to apply Theorem~\ref{sale} for weighted Dirichlet spaces, we will use the following process. 

First, it suffices to prove the lower estimate with $({\mathcal D}^{\, 2}_\alpha)^\ast$ instead of ${\mathcal D}^{\, 2}_\alpha$, where: 
\begin{equation} 
({\mathcal D}^{\, 2}_\alpha)^\ast = \{ f \in {\mathcal D}^{\, 2}_\alpha \tq f (0) = 0 \} 
\end{equation} 
is the hyperplane of ${\mathcal D}^{\, 2}_\alpha$ of functions vanishing at $0$. 
\smallskip

The derivation $\Delta$ is by definition a unitary operator from $(\mathcal{D}^{\, 2}_\alpha)^\ast$ onto  $\mathfrak{B}^2_\alpha$. For any symbol 
$\phi$ vanishing at $0$,  we have the following diagram, where $w = \phi '$:
\begin{equation} 
(\mathcal{D}^{\, 2}_\alpha)^\ast \xrightarrow{\phantom{M} \Delta \phantom{M}} \mathfrak{B}^2_\alpha 
\xrightarrow{\, M_w C_\phi\, } \mathfrak{B}^2_\alpha \xrightarrow{\ \Delta^{ - 1}\,} (\mathcal{D}^{\, 2}_\alpha)^\ast \, ,
\end{equation} 
with obviously (since $\phi ' \, (f ' \circ \phi) = (f \circ \phi) '$):
\begin{displaymath} 
C_\phi^{(\mathcal{D}^{\, 2}_\alpha)^\ast} = \Delta^{- 1} (M_w C_\phi) \, \Delta \, ,
\end{displaymath} 
which shows that $C_{\varphi}^{(\mathcal{D}^{\, 2}_\alpha)^\ast}$, acting on the delicate space $(\mathcal{D}^{\, 2}_\alpha)^\ast$ is unitarily equivalent 
to the weighted composition operator $M_w C_\phi$ acting on the more robust space $ \mathfrak{B}^2_\alpha$ (in that 
$\mathcal{M} (\mathfrak{B}^2_\alpha) = H^\infty$). 
Moreover,  $C_\phi^{(\mathcal{D}^{\, 2}_\alpha)^\ast}$ and $M_w C_\phi \colon \mathfrak{B}^2_\alpha \to \mathfrak{B}^2_\alpha$ have the same 
approximation numbers. 
\goodbreak

\subsection {The cusp map on weighted Dirichlet spaces} \label{cusp}

First, we recall the definition of the cusp map $\chi$. We begin by defining:
\begin{equation} \label{chi-zero}
\chi_0 (z) = \frac{\displaystyle \Big( \frac{z - i}{i z - 1} \Big)^{1/2} - i} {\displaystyle - i \, \Big( \frac{z - i}{i z - 1} \Big)^{1/2} + 1} \, \cdot 
\end{equation}
That defines a conformal mapping from $\D$ onto the right half-disk 
\begin{displaymath} 
D = \{z \in \D \tq \Re z > 0\} 
\end{displaymath} 
such that $\chi_0 (1) = 0$, $\chi_0 (- 1) = 1$, 
$\chi_0 (i) = - i$, $\chi_0 (- i) = i$, and $\chi_0 (0) = \sqrt{2} - 1$. Then we set: 
\begin{equation} \label{chi-1-2-3} 
\chi_1 (z) = \log \chi_0 (z), \quad \chi_2 (z) = - \frac{2}{\pi}\, \chi_1 (z) + 1, \quad \chi_3 (z) = \frac{a}{\chi_2 (z)}  \, \raise 1pt \hbox{,} 
\end{equation}
and finally:
\begin{equation} \label{chi} 
\chi (z) = 1 - \chi_3 (z) \, ,
\end{equation}
where:
\begin{equation} \label{definition de a}
a = 1 - \frac{2}{\pi} \log (\sqrt{2} - 1) = 1.56 \ldots \, \in (1, 2) 
\end{equation} 
is chosen in order that $\chi (0) = 0$. The image $\Omega$ of the (univalent) cusp map $\chi$ is formed by the intersection of the  inside of the disk 
$D \big(1 - \frac{a}{2} \raise 1pt \hbox{,} \frac{a}{2} \big)$ and the outside of the two closed disks  
$\overline{D \big(1 + \frac{i a}{2} \raise 1pt \hbox{,} \frac{a}{2} \big)}$ and $\overline{D \big(1  - \frac{ i a}{2} \raise 1pt \hbox{,} \frac{a}{2} \big)}$. 
\smallskip

Since $\chi$ is injective, it follows from Zorboska's characterization in \cite{Zorb} (see also \cite[Section~6.12]{Shapiro}) that the composition operator 
$C_\chi$ is bounded on ${\mathcal D}^{\, 2}_\alpha$ for $\alpha \geq 0$. In particular $\chi \in {\mathcal D}^{\, 2}_\alpha$. 
\smallskip

\begin{theorem} \label{diri} 
Let $\chi$ be the cusp map acting  on the Dirichlet space $\mathcal{D}^{\, 2}_\alpha$. Then, for some positive constants $b'_\alpha > b_\alpha > 0$, 
depending only on $\alpha$, we have, for all $n \geq 1$:
\begin{equation} \label{fri} 
\qquad  \qquad  \e^{- b'_\alpha \, n/ \log n}\lesssim a_n (C_\chi) \lesssim \e^{- b_\alpha \, n/ \log n} \qquad \text{for } \alpha > 0 \, ,
\end{equation}
and
\begin{equation} \label{approx numb Dirichlet} 
\qquad \qquad\qquad  \e^{- b'_0 \sqrt n} \lesssim a_n (C_\chi) \lesssim \e^{- b_0 \sqrt n} \qquad\quad \text{for } \alpha = 0 \, .
\end{equation}
\end{theorem} 
\smallskip

Actually, the proof shows that, for $\alpha > 0$, the constant $b_\alpha$ can be chosen as $c\min(1,\alpha)$, and $b'_\alpha$ can be chosen as 
$c' \max(1, \alpha)$, where $c$, $c'$ are absolute positive constants.
\smallskip

Note that the case $\alpha < 0$ is not relevant since then composition operators $C_\phi$ that are compact on ${\mathcal D}^{\, 2}_\alpha$ must satisfy 
$\| \phi \|_\infty < 1$ (\cite{Shapiro-PAMS}).
\smallskip

The estimates \eqref{approx numb Dirichlet} was first proved in \cite[Theorem~3.1]{LLQR-Arkiv}, with ad hoc methods. We give here a more transparent 
proof of the lower bound, based on weighted composition operators acting on ${\mathfrak B}^2_\alpha$, and that works for all $\alpha \geq 0$. Here the 
weight is $\chi '$. Since $\chi \in {\mathcal D}^{\, 2}_\alpha$, we have $\chi ' \in {\mathfrak B}^2_\alpha$. 
\smallskip

We have the following estimates (the first one was given in \cite[Lemma~4.2]{LQR-Fennicae}).
\begin{lemma} \label{ajout} 
When $r \to 1^{-}$, it holds:
\begin{equation} \label{jeudi} 
1 - \chi (r) \approx \frac{1}{\log [1 / (1 - r)]} \qquad \text{and} \qquad \chi ' (r) \approx \frac{1}{(1 - r)\log^{2} [1 / (1 - r)]} \, \cdot
\end{equation}
\end{lemma}
\begin{proof} 
For $r \in (0,1)$, we have (\cite[Lemma~4.2]{LQR-Fennicae}):
\begin{displaymath} 
\chi_0 (r) = \tan \bigg[ \frac{1}{2} \, \arctan \bigg(\frac{1 - r}{1 + r}\bigg) \bigg] = \tan \Big( \frac{\pi}{8} - \frac{1}{2}\, \arctan r \Big) \, ;
\end{displaymath} 
hence $\chi_0 (r) \approx 1 - r$ and:
\begin{displaymath} 
\chi_0 ' (r) = - \frac{1 +[\chi_0 (r)]^2}{2(1 + r^2)} \, \cdot
\end{displaymath} 
Using the definitions \eqref{chi-1-2-3} and \eqref{chi}, we get: 
\begin{displaymath} 
\chi_2 ' = - \frac{2}{\pi} \, \frac{\chi_0 '}{\chi_0} \quad \text{and} \quad \chi_3 ' = - \frac{a \, \chi_2 '}{\chi_2^2} \, \cdot 
\end{displaymath} 
So we have, when $r \to 1^{-}$:
\begin{displaymath} 
\chi_2 (r) \approx \log [1 / (1 - r)]  
\end{displaymath} 
and:
\begin{displaymath} 
\chi '_2 (r) = \bigg(\frac{1 +[\chi_0 (r)]^2}{\pi(1 + r^2)} \bigg) \, \frac{1}{\chi_0 (r)} \approx \frac{1}{\chi_0 (r)} 
\approx \frac{1}{1 - r} \, \cdot 
\end{displaymath} 
The result follows.
\end{proof}
\begin{proof} [Proof of Theorem~\ref{diri}] \hfill \par
\smallskip

\noindent {\sl Proof of the lower bound.}

We  choose $0 < u_j < 1$ so as to get via \eqref{jeudi}, with $v_j = \chi (u_j)$ and $\varepsilon > 0$ to be adjusted later:
\begin{equation} \label{venerdi} 
1 - v_j = \e^{- j \varepsilon} \, ;
\end{equation}
hence:
\begin{equation} 
\log \Big(\frac{1}{1 - u_j} \Big) \approx \e^{j\varepsilon} \, .
\end{equation} 

The interpolation constant $I_v$ of the sequence $v = (v_j)_j$ satisfies $I_v \lesssim 1 / \delta_v^2$, where $\delta_v$ is its Carleson constant (see 
\cite[Chapter~VII, Theorem~1.1]{Garnett}). Since $\frac{1 - v_{j + 1}}{1 - v_j} = \e^{- \eps}$, we have (\cite[Lemma~6.4]{LQR-JAT}), for some 
positive constants $c_1$, $c_2$:
\begin{displaymath} 
\delta_v \geq \exp \Big( - \frac{c_1}{1 - \e^{- \eps}} \Big) \geq \e^{ - c_2 / \eps} \, ;
\end{displaymath} 
hence, with $c_0 = 2 \, c_2$:
\begin{displaymath} 
I_v \lesssim \e^{\, c_0 / \eps} 
\end{displaymath} 
(where the implicit constant does not depend on $\eps$).
\smallskip\goodbreak

The reproducing kernel of the Bergman space ${\mathfrak B}^2_\alpha$ satisfies:
\begin{displaymath} 
\| K_a \| = \frac{1}{(1 - |a|^2)^{1 + \alpha/2}} \approx \frac{1}{(1 - |a|)^{1 + \alpha/2}} \, \cdot
\end{displaymath} 
Using \eqref{jeudi} and \eqref{venerdi}, we get, for $1 \leq j \leq n$:
\begin{align*}
|\chi '(u_j)| \, \frac{\| K_{v_j} \|}{\| K_{u_j} \|}
& \gtrsim \frac{1}{(1 - u_j) \, \log^{2} [1 / (1 - u_{j})]} \, \frac{(1 - u_j)^{1 + \alpha/2}}{(1 - v_j)^{1 + \alpha/2}} \\ 
& = \frac{1}{\log^{2} [1 / (1 - u_{j})]} \, \frac{(1 - u_j)^{\alpha/2}}{\ (1 - v_j)^{1 + \alpha/2}} \\
& \approx \frac{1}{\e^{ 2 j \eps}} \, \frac{\exp ( - \frac{\alpha}{2} \, \e^{j \eps})}{\e^{- j \eps (1 + \alpha / 2)}} 
= \exp \bigg( - \bigg[ \frac{\alpha}{2} \,  \e^{j \eps} + j \eps \Big( 1 - \frac{\alpha}{2} \Big) \bigg] \bigg) \\
& \gtrsim \exp \bigg( - \bigg[ \frac{\alpha}{2} \,  \e^{n \eps} + n \eps \Big( 1 - \frac{\alpha}{2} \Big) \bigg]\bigg) \, ,
\end{align*}
since the function $t \mapsto \frac{\alpha}{2} \, \e^t + t \, \big( 1 - \frac{\alpha}{2} \big)$ is increasing; in fact, its derivative is positive.
\smallskip

Theorem~\ref{sale} with $w = \chi '$ gives:
\begin{equation} \label{sharp estimate}
a_n^{{\mathcal B}_\alpha^{\, 2}}(M_{\chi '} C_\chi) \gtrsim 
\exp \bigg( - \bigg[ \frac{\alpha}{2} \,  \e^{n \eps} + n \eps \Big( 1 - \frac{\alpha}{2}\Big)  + \frac{2 \, c_0 }{\eps} \bigg] \bigg) \, .
\end{equation} 
\smallskip 

{\sl Case $\alpha = 0$}. In this case, we have:
\begin{displaymath}
a_n^{{\mathfrak B}_0^2}(M_{\chi '} C_\chi) \gtrsim \exp \bigg( - \bigg[n \eps + \frac{2 \, c_0 }{\eps} \bigg] \bigg) \, .
\end{displaymath}
Taking $\eps = 1 / \sqrt n$, we get:
\begin{displaymath}
a^{{\mathcal D}_0^{\, 2}}_n (C_\chi) = a^{{\mathfrak B}_0^2}_n (M_{\chi '} C_\chi) \gtrsim \e^{ - c \sqrt{n}} 
\end{displaymath}
for some positive absolute constant $c$.
\smallskip

{\sl Case $\alpha > 0$}. We take $\eps = \frac{1}{n} \, \log \big( \frac{n}{\log n} \big)$ and we get:
\begin{displaymath} 
a^{{\mathfrak B}_\alpha^2}_n (M_{\chi '} C_\chi)\gtrsim\exp \bigg( - \bigg[ \frac{\alpha}{2} \,  \frac{n}{\log(n)} +\log\Big(\frac{n}{\log(n)}\Big)  
+ 2 \, c_0 \frac{n}{\log(n)-\log(\log(n))}\bigg]\bigg) \, ,
\end{displaymath}
Since $\log\big(\frac{n}{\log(n)}\big)  + 2 \, c_0 \frac{n}{\log(n)-\log(\log(n))} = {\rm O}\, \big( \frac{n}{\log(n)}\big)$, we get:
\begin{displaymath} 
a^{{\mathcal D}_\alpha^{\, 2}}_n (C_\chi)\gtrsim \exp (- b'_\alpha  \, n / \log n) \, ,
\end{displaymath}
with $b'_\alpha = c' \max (\alpha, 1)$, where $c'$ is some positive absolute constant.
\medskip\goodbreak

\noindent {\sl Proof of the upper bound.} 

For $\alpha = 0$, the upper bound is proved in \cite[Theorem~3.1]{LLQR-Arkiv}; for $\alpha > 0$, we will follow that proof, with the same notation, but 
with the weighted Nevanlinna counting function $N_{\chi, \alpha}$ instead of the counting function $n_\chi$. Recall that the weighted Nevanlinna counting 
function of the analytic function $\phi \colon \D \to \D$ is:
\begin{displaymath} 
N_{\phi, \alpha} (w) = \sum_{\phi (z) = w} (1 - |z|^2)^\alpha \, .
\end{displaymath} 
Note that this definition is slightly different, though equivalent, from that given in \eqref{weighted Nevanlinna}, but it is more convenient here.

Since the cusp map $\chi$ is univalent, we have:
\begin{equation} \label{thtis} 
\qquad \!\! N_{\chi, \alpha} (w) = 
\left\{
\begin{array} {ll}
(1 - |\chi^{- 1} (w)|^2)^\alpha & \qquad \text{for } w \in \chi (\D) \, , \\ 
\ 0 & \qquad \text{otherwise.} 
\end{array}
\right.
\end{equation}

The Schwarz lemma gives $N_{\chi, \alpha} (w) \leq (1 - |w|^2)^\alpha$, but the following lemma gives the better estimate:
\begin{equation} \label{Nevanlinna cusp}
\qquad \qquad \qquad N_{\chi, \alpha} (w) \lesssim \e^{- c_0 \alpha / (1 - |w|)} \, , \quad \text{for } w \in \chi (\D)
\end{equation} 
\begin{lemma} \label{hope}  
We have:
\begin{displaymath} 
\qquad \quad 1 -| \chi (z) |\gtrsim \frac{1}{\log [1/|1 - z|]} \quad  \text{for  all } z \in \D \, ,
\end{displaymath} 
so that, for some positive constant $c_0$:
\begin{displaymath} 
\qquad 1 - |\chi^{- 1} (w)| \leq |1 - \chi^{- 1} (w)| \lesssim  \exp \Big( - \frac{c_0}{1 - |w|} \Big) \quad \text{for all } w \in \chi (\D) \, .
\end{displaymath} 
\end{lemma}
\begin{proof} [Proof of the lemma] 
It suffices to look at the neighborhood of $1$ since outside, the two functions are continuous and do not vanish. Setting $u = 1 - z$, an easy computation with 
Taylor expansions gives that $\frac{z - i}{i z - 1} = - 1 + i u + {\rm o}\, (u)$, as $u \to 0$, so  
$\big( \frac{z - i}{i z - 1} \big)^{1 /2} = i [1 -  i u / 2 + {\rm o}\, (u) ] = i +  u / 2 + {\rm o}\, (u)$, and (recall \eqref{chi-zero} and \eqref{chi-1-2-3}):
\begin{displaymath} 
\qquad \quad \chi_0 (1 - u) = \frac{u}{4} + {\rm o}\, (u) \quad \text{as } u \to 0 \, ;
\end{displaymath} 
and $\displaystyle |1 - \chi(z)|= |\chi_3 (z)|\approx \frac{1}{|\log(\chi_0 (z))|}\gtrsim\frac{1}{\log (1 / |1 - z|)}\,\cdot$
\smallskip

Finally, since the cusp is contained in an angular sector, there exists some $\delta > 0$ such that  $1 - |\chi(z)| \ge \delta |1 - \chi(z)|$ for every 
$z\in\D$. The result follows. 
\end{proof}

That allows to get the following estimate. 
\goodbreak

\begin{lemma} \label{1} 
We have:
\begin{displaymath} 
\| \chi^n \|_{\mathcal{D}^{\, 2}_{\alpha}}^2 \lesssim n^2 \, \e^{- \sqrt{ 2 c_0 \alpha \, n}} \, .
\end{displaymath} 
\end{lemma}
\begin{proof}
We have, since $\chi (0) = 0$:
\begin{align*}
\| \chi^n\|_{{\mathcal D}^{\, 2}_\alpha}^2 
& = \int_\D | n \chi^{n - 1} (z) \, \chi ' (z) |^2 \, d A_\alpha (z) \\ 
& = \int_{\chi (\D)} n^2 |w|^{2 n - 2} N_{\chi, \alpha} (w) \, dA (w) \\
& \leq \int_{\chi (\D)} n^2 |w|^{2 n - 2} \, \e^{- c_0 \alpha / (1 - |w|)} \, dA (w) \\
& \leq n^2 (1 - h)^{2 n - 2} + n^2 \int_{{\chi (\D)} \cap \{|w| \geq 1 - h \}} \e^{- c_0 \alpha / (1 - |w|)} \, dA (w) \\
& \lesssim n^2 \, \e^{- 2 n h} + n^2 \, \e^{- c_0 \alpha / h} \, ,
\end{align*}
Choosing 
$h = \sqrt {c_0 \alpha / 2n}$ gives:
\begin{displaymath} 
\| \chi^n\|_{{\mathcal D}^{\, 2}_\alpha}^2 \lesssim n^2 \, \e^{- \sqrt{2 c_0 \alpha \, n}} \, . \qedhere
\end{displaymath} 
\end{proof}

For $f (z) = \sum_{n = 0}^\infty c_n z^n$, we define:
\begin{displaymath} 
(S_N f) (z) = \sum_{n = 0}^N c_n z^n \, .
\end{displaymath} 
As a consequence of Lemma~\ref{1}, we have the following majorization.
\begin{lemma} \label{2} 
We have:
\begin{equation} \label{eq-lemma 2}
\| C_\chi - C_\chi S_N \|_{\mathcal{D}^{\, 2}_{\alpha}} \lesssim  N^{\frac{3 + 2 \alpha}{4}} \, \e^{- \sqrt{2 c_0 \alpha \, N}} \, .
\end{equation} 
\end{lemma}
\begin{proof}
It suffices to use the Hilbert-Schmidt norm and Lemma~\ref{1}:
\begin{align*} 
\| C_\chi - C_\chi S_N \|_{\mathcal{D}^{\, 2}_{\alpha}}^2 
& \leq \| C_\chi - C_\chi S_N \|_{\rm HS}^2 \approx \sum_{n > N} \frac{\| \chi^n\|_{\mathcal{D}^{\, 2}_{\alpha}}^2}{n^{1 - \alpha}} \\
& \lesssim \sum_{n > N} n^{1 + \alpha} \, \e^{- \sqrt{2 c_0 \alpha \, n}} 
\approx  N^{\frac{3}{2} + \alpha} \, \e^{- \sqrt{2 c_0 \alpha \, N}} \, . \qedhere
\end{align*} 
\end{proof}
\begin{lemma} \label{3} 
Let $J$ be the canonical injection $J \colon H^2 \to L^2 (\mu)$ with $d \mu = N_{\chi, \alpha} \, dA$. Then:
\begin{equation} \label{eq-lemma 3}
a_n (C_\chi S_N) \lesssim N^{\frac{1+\alpha}{2}} a_n (J) \, .
\end{equation} 
\end{lemma}
\begin{proof}
Let $f \in ({\mathcal D}^{\, 2}_\alpha)^\ast$ and write $f (z) = \sum_{j = 1}^\infty c_j z^j$, we have:
\begin{displaymath} 
\| C_\chi S_N f \|_{{\mathcal D}^{\, 2}_\alpha}^2 = \int_\D \bigg| \sum_{j = 1}^N j c_j w^{j - 1} \bigg|^2 \, N_{\chi, \alpha} (w) \, dA (w) 
= \int_\D | (\Delta_N f) (w) |^2 \, d\mu (w) \, ,
\end{displaymath} 
where $\Delta_N \colon (\mathcal{D}^{\, 2}_\alpha)^\ast \to H^2$ is defined by:
\begin{displaymath} 
(\Delta_N f) (w) = \sum_{j = 1}^N j c_j w^{j - 1} \, .
\end{displaymath} 
We hence have $\| C_\chi S_N f \|_{{\mathcal D}^{\, 2}_\alpha} = \| J \Delta_N f \|_{L^2 (\mu)}$ for all $f \in ({\mathcal D}^{\, 2}_\alpha)^\ast$. 
It follows that there exists a contraction $T_N \colon L^2 (\mu) \to ({\mathcal D}^{\, 2}_\alpha)^\ast$ such that:
\begin{displaymath} 
C_\chi S_N = T_N J \Delta_N \, .
\end{displaymath} 
Now:
\begin{displaymath} 
\| \Delta_N f \|_{H^2}^2 = \sum_{j = 0}^N j^{1 + \alpha} j^{1 - \alpha} |c_j|^2 
\leq N^{1 + \alpha} \sum_{j = 0}^N j^{1 - \alpha} |c_j|^2 \leq N^{1 + \alpha} \| f \|_{{\mathcal D}^{\, 2}_\alpha}^2 \, ;
\end{displaymath} 
hence, by the ideal property of approximation numbers:
\begin{displaymath} 
a_n (C_\chi S_N) \approx a_n \big ( (C_\chi S_N)_{\mid ({\mathcal D}^{\, 2}_\alpha)^\ast} \big) \leq \| T_N \| \, \| \Delta_N \| \, a_n (J) 
\lesssim N^{\frac{\alpha + 1}{2}} \, a_n (J) \,. \qedhere
\end{displaymath} 
\begin{proposition} \label{4} 
Let $\alpha > 0$ and $J$ be the canonical injection $J \colon H^2 \to L^2 (\mu)$ with $d \mu = N_{\chi, \alpha} \, dA$. Then, for some absolute positive 
constant $c$:
\begin{equation} \label{eq-prop 4}
a_n (J) \lesssim\exp\Big({ - c \,\min (\alpha, 1) \, \frac{n}{\log n}}\Big) \,.
\end{equation} 
\end{proposition}
\begin{proof}
We use a modification of the Blaschke product of \cite[page~168]{LLQR-Arkiv}, as follows. 
Let $r = [\log_2 n]$ be the greatest integer $< \log_2 n$, where $\log_2$ is the binary logarithm, and $B_0$ be the Blaschke product with simple zeros at the 
points:
\begin{displaymath} 
\qquad \quad z_j = 1 - 2^{- j} \, , \qquad  1 \leq j \leq r \, ,
\end{displaymath} 
and we consider the Blaschke product $B = B_0^n$. 

Let $E = B H^2$, which is a subspace of $H^2$ of codimension $n \, [\log_2 n]$. 

We have, by the Carleson embedding theorem for $H^2$ (see \cite[Lemma~2.4]{LLQR-Israel}):
\begin{equation} \label{again} 
\| J_{\mid E} \|^2 \lesssim \sup_{\substack{0 < h < 1 \\ \xi \in \T}} \frac{1}{h} \int_{ S (\xi, h) \cap \Omega} |B|^2 N_{\chi, \alpha} \, dA \, ,
\end{equation}
where $S (\xi, h) = \{ z \in \D \tq |z - \xi| < h\}$ and:
\begin{equation} 
\Omega = \chi (\D) \, . 
\end{equation} 

Note that $A [ S (\xi, h) \cap \Omega] \lesssim h^3$ since the area of ${\chi (\D)} \cap \{|w| \geq 1 - h \}$ is $\approx h^3$; in fact this 
set is delimited at the cuspidal point $1$ by two circular arcs.
\smallskip

Now we majorize the right-hand side of \eqref{again}.  
For that, we first note that, since $\Omega$ is contained in an angular sector, there is an absolute positive constant $\delta_0$ such that 
$1 - |w| \geq \delta_0 |1 - w|$ for all $w \in \Omega$. 
Hence if $w \in S (\xi, h) \cap \Omega$, we have:
\begin{displaymath} 
\delta_0 \, |1 - w| \leq 1 - |w| \leq |\xi - w|  < h \, ,
\end{displaymath} 
and $w \in S (1, h/\delta_0)$. Hence $S (\xi, h) \cap \Omega \subseteq S (1, h/ \delta_0) \cap \Omega$.
\smallskip\goodbreak

Moreover, we may assume that $h = \delta_0 2^{- l}$ and we separate two cases. 
\smallskip\goodbreak

$\bullet$  $l \geq r$. 

We simply majorize $|B|$ by $1$. Lemma~\ref{hope} and \eqref{Nevanlinna cusp} lead to:
\begin{align*}
\frac{1}{h} \int_{S (\xi, h) \cap \Omega} |B|^2 N_{\chi, \alpha} \, dA 
& \lesssim \frac{1}{h} \int_{S (\xi, h) \cap \Omega} \exp (- \alpha \, c_0 / h) \, dA \\ 
& \lesssim \frac{1}{h} \, \e^{- \alpha c_0 / h} \, A [ S (\xi, h) \cap \Omega] \approx h^2 \, \e^{- \alpha c_0 / h} \\
& \leq \e^{- \alpha \, (c_0 / \delta_0) \, 2^l} \leq \e^{- \alpha \, (c_0 / \delta_0) \, 2^r} \leq \e^{- \alpha \, (c_0 / 2 \delta_0)  n} \, .
\end{align*}
\smallskip\goodbreak

$\bullet$  $l < r$. 

We write:
\begin{align*} 
\int_{S (\xi, h) \cap \Omega} |B|^2 N_{\chi, \alpha} \, dA 
& \leq \int_{S (1, h/ \delta_0) \cap \Omega} |B|^2 N_{\chi, \alpha} \, dA \\
& = \int_{S (1, 2^{- r}) \cap \Omega}|B|^2 N_{\chi, \alpha} \, dA \\
& \qquad \qquad \qquad + \int_{\{w \in \Omega \tq 2^{- r} \leq | w - 1| < 2^{- l} \}} |B|^2 N_{\chi, \alpha} \, dA \, .
\end{align*} 

By the previous case, the first integral in the right-hand side, divided by $h$, is $\lesssim \e^{- \alpha \, (c_0 / 2 \delta_0)  n}$. 

Now, if $2^{- r} \leq | w - 1| < 2^{- l}$, there is some $j \in \{l + 1, \ldots, r\}$ such that 
$2^{- j} \leq |w - 1| < 2^{- j + 1}$. Then:
\begin{displaymath} 
|w - z_j| \leq |w - 1| + |1 - z_j| \leq 2^{- j + 1} + 2^{- j} = 3\, . \, 2^{- j} \, ;
\end{displaymath} 
hence, since $2^{- j} \leq (1 / \delta_0) \, (1 - |w|)$ and $1 - |z_j| = 1 - z_j = 2^{- j}$, we have, with $M = \max (3, 1/ \delta_0)$:
\begin{displaymath} 
|w - z_j| \leq M \, \min (1 - |w|, 1 - |z_j|) \, .
\end{displaymath} 
Therefore \cite[Lemma~2.3]{LLQR-Israel} shows that the modulus of the $j$-th factor of $B_0 (w)$ is less or equal than $\kappa = M / \sqrt{M^2 + 1} < 1$. 
It follows that $|B (w)| = |B_0 (w)|^n \leq \kappa^n$. 
Using $N_{\chi, \alpha} (w) \leq 1$ we obtain:
\begin{displaymath} 
\frac{1}{h} \int_{\{w \in \Omega \tq 2^{- r} \leq | w - 1| < 2^{- l} \}} |B|^2 N_{\chi, \alpha}(w) \, dA 
\lesssim \frac{1}{h} \, A [S (\xi, h) \cap \Omega] \, \kappa^{2 n} \lesssim  \kappa^{2 n} = \e^{- \eps_0 n}
\end{displaymath} 
for some absolute constant $\eps_0 > 0$.

We get:
\begin{equation} \label{II}
\frac{1}{h} \int_{S (\xi, h) \cap \Omega} |B|^2 N_{\chi, \alpha} \, dA \lesssim \e^{- \alpha \, (c_0 / 2\delta_0) n} + \e^{- \eps_0 n} 
\lesssim \e^{- c_1 \min(\alpha,1)\, n} \, ,
\end{equation} 
where $c_1 > 0$ does not depend on $\alpha$.  
\smallskip

In either case, we obtain,
\begin{displaymath} 
\| J_{\mid E} \| \lesssim \e^{- c_2  \min(\alpha,1)\, n} \, .
\end{displaymath} 
That means that the Gelfand number $c_{n\, [\log_2 n]} (J)$ is $\lesssim \e^{- c_2 \min(\alpha,1)\, n}$. Since the Gelfand numbers are the same as the 
approximation numbers on Hilbert spaces, we get:
\begin{displaymath} 
a_{n \, [\log_2 n]} (J) \lesssim \e^{- c_2  \min(\alpha, 1)\, n} ,
\end{displaymath} 
or, making change of variables in the indices:
\begin{displaymath} 
a_n (J) \lesssim \exp\Big({ - c_3 \,\min(\alpha,1) \, \frac{n}{\log n}}\Big) \, \raise 1pt \hbox{,}
\end{displaymath} 
as claimed.
\end{proof}

\noindent {\sl End of the proof of the upper bound.} For every operator $R \colon {\mathcal D}^{\, 2}_\alpha \to {\mathcal D}^{\, 2}_\alpha$ with 
rank $< n$, we have:
\begin{displaymath} 
\| C_\chi - R \| \leq \| C_\chi - C_\chi S_N \| + \| C_\chi S_N - R\| \, ;
\end{displaymath} 
so:
\begin{displaymath} 
a_n (C_\chi) \leq \| C_\chi - C_\chi S_N \| + a_n (C_\chi S_N) \, .
\end{displaymath} 
Using Lemma~\ref{2}, Lemma~\ref{3} and Proposition~\ref{4}, we obtain:
\begin{align*} 
a_n (C_\chi) 
& \lesssim N^{\frac{3 + 2 \alpha}{4}} \, \e^{- \sqrt{2 c_0 \alpha \, N}}  + 
N^{\frac{\alpha + 1}{2}} \e^{ - c_3 \,\min(\alpha, 1) \, n / \log n}  \\
& \lesssim N^{\frac{3 + 2 \alpha}{4}} \, \big( \e^{- \sqrt{2 c_0 \alpha \, N}} + \e^{- c_3 \min (\alpha, 1) \, n / \log n} \big) \, .
\end{align*} 

a) Suppose first that $\alpha \le 1$. Then:
\begin{displaymath} 
a_n (C_\chi) \lesssim  N^{\frac{3 + 2 \alpha}{4}} \, \big( \e^{- \sqrt{2 c_0 \alpha \, N}} + \e^{- c_3 \alpha \, n / \log n} \big) \, .
\end{displaymath} 

Choosing $N$ as the integral part of $\alpha (c_3^2/ 2 c_0) \, (n / \log n)^2$, we get:
\begin{displaymath} 
a_n (C_\chi) \lesssim \bigg(\frac{n}{\log n} \bigg)^{\frac{3 + 2 \alpha}{2}} \, \e^{- c_3 \alpha n / \log n} 
\lesssim \bigg(\frac{n}{\log n} \bigg)^{\frac{5}{2}} \, \e^{- c_3 \alpha n / \log n} 
\lesssim \e^{- c'_3 \alpha n / \log n} \, ,
\end{displaymath} 
for another absolute constant $c'_3 < c_3$.

b) For $\alpha > 1$, we have:
\begin{displaymath} 
a_n (C_\chi) \lesssim  N^{\frac{3 + 2 \alpha}{4}} \, \big( \e^{- \sqrt{2 c_0 \alpha \, N}} + \e^{- c_3 \, n / \log n} \big) \, .
\end{displaymath} 
Choosing for $N$ the integral part of $(1/ \alpha) (c_3^2/ 2 c_0) (n / \log n)^2$, we get:
\begin{displaymath} 
a_n (C_\chi) \lesssim  \bigg( \frac{n}{\log n} \bigg)^{\frac{3 + 2 \alpha}{2}} \, \e^{- c_3 n / \log n} \, .
\end{displaymath} 

However, the term $(n /\log n )^{\frac{3 + 2 \alpha}{2}}$ tends to infinity when $\alpha$ tends to infinity (even if the implicit constants in the inequalities 
depend on $\alpha$). In order to have a better estimate, we are going to follow a different way.
\smallskip

We recall that  $\mathcal {D}_{1}^{\, 2}=H^2$. We now use Theorem~\ref{theo compos op}, which is licit as indicated in 
the remarks following the statement of this theorem. Using the previously treated case, we obtain  that:
\begin{displaymath} 
a_{2n}^{{\mathcal D}_\alpha^{\, 2}} (C_\chi)\le  \sqrt { {a_1^{H^2} (C_\chi})\,a_n^{H^2} (C_\chi})\lesssim \exp\Big({ - c\,\frac{n}{\log n}}\Big) 
\, \raise 1pt \hbox{,}
\end{displaymath} 
hence, rescaling on one hand and using the monotony of the sequence of the approximation numbers on the other hand, we get:
\begin{displaymath} 
a_{n}^{{\mathcal D}_\alpha^{\, 2}} (C_\chi)\lesssim \exp\Big({ - c'\,\frac{n}{\log n}}\Big) \, \raise 1pt \hbox{,}
\end{displaymath} 
where the underlying constants do not depend on $\alpha$.
\end{proof}
\smallskip

That ends the proof of Theorem~\ref{diri}.
\end{proof}
\goodbreak

\noindent {\bf Remark~1.} Actually, for $\alpha > 0$, the formula \eqref{sharp estimate} gives that:
\begin{equation}
a_n^{{\mathcal D}_\alpha^{\, 2}} (C_\chi) \gtrsim 
\exp \bigg( - \bigg[ \frac{\alpha}{2} \,  \e^{n \eps} + n \eps \Big( 1 - \frac{\alpha}{2}\Big)  + \frac{2 \, c_0 }{\eps} \bigg] \bigg) \, .
\end{equation}

Taking $ \eps = 1 / \sqrt n$, we get, for $0 < \alpha < 2$, with $c_1 = 1 + 2 \, c_0$, this ``bad'' estimate:
\begin{displaymath}
a_n^{{\mathcal D}_\alpha^{\, 2}} (C_\chi) \gtrsim 
\exp \bigg( - \bigg[ \frac{\alpha}{2} \,  \e^{\sqrt n} + c_1 \, \sqrt{n} \bigg] \bigg) \, .
\end{displaymath}

Nevertheless, Theorem~\ref{theo compos op}, $3)$ $a)$, gives:
\begin{displaymath}
a_n^{{\mathcal D}_0^{\, 2}} (C_\chi) \gtrsim\Big(a_{2 n}^{{\mathcal D}_\alpha^{\, 2}} (C_\chi)\Big)^2 \gtrsim 
\exp \bigg( - \bigg[ \alpha \,  \e^{\sqrt {2 n}} + 2c_1 \, \sqrt{2 n} \bigg] \bigg) \, .
\end{displaymath}
Note that, despite we did not explicit them, the implicit constants in these inequalities are $\approx 2^{- \alpha / 2}$; so, letting $\alpha$ tend to $0$, we obtain, 
with $c = 2^\frac{3}{2}\, c_1$:
\begin{displaymath}
a_n^{{\mathcal D}_0^{\, 2}} (C_\chi) \gtrsim \e^{- c \sqrt{n}} \, ,
\end{displaymath}
explaining the jump between the cases $\alpha > 0$ and $\alpha = 0$.
\smallskip\goodbreak

\noindent {\bf Remark~2.} When $\alpha\rightarrow0^+$, the behavior both of the upper and the lower estimates remains quite far from the one in the case 
$\alpha = 0$. It would be interesting to get a better control on both sides relatively to $\alpha$ to understand the breaking point between the case $\alpha > 0$ 
and the case $\alpha = 0$. Very likely, it would require a different viewpoint and different methods to estimate approximation numbers. 

\subsection{Lens maps for weighted Dirichlet spaces} \label{sec:lens} 

In this section, we consider lens maps (see \cite[page~27]{Shapiro-livre}. Let us recall that for $0 < \theta < 1$, the lens map $\lambda_\theta$ of parameter 
$\theta$ is the map from $\D$ into $\D$ defined by:
\begin{equation} \label{definition lens}
\qquad \lambda_{\theta} (z) = \frac{(1 + z)^\theta - (1 - z)^\theta }{(1 + z)^\theta + (1 - z)^\theta} \, \cdot
\end{equation}

It is a conformal map obtained by sending $\D$ onto the right-half plane, then taking the power $\theta$, and going back to $\D$.
\smallskip

Since $\lambda_\theta$ is univalent, it follows from \cite[Theorem~1]{Zorb} that the associated composition operator $C_{\lambda_\theta}$ is bounded on 
${\mathcal D}^{\, 2}_\alpha$ for all $\alpha \geq 0$.

\begin{theorem} \label{th-lens} 
Let $0 < \theta < 1$ and $\lambda_\theta$ be the lens map of parameter $\theta$. Then the composition operator $C_{\lambda_\theta}$ is not compact on 
${\mathcal D}^{\, 2} = {\mathcal D}^{\, 2}_0$; but for all $\alpha > 0$, $C_{\lambda_\theta}$ is compact on 
$\mathcal{D}_{\alpha}^{\, 2}$, and moreover there are positive constants $b > b' > 0$, depending only on $\theta$ and $\alpha$, 
such that, for all $n \geq 1$: 
 \begin{equation} \label{mon} 
\e^{- b \sqrt{n}} \lesssim a_n (C_{\lambda_\theta} ) \lesssim \e^{- b' \sqrt{n}} \, .
\end{equation}

In particular, for $\alpha > 0$, $C_{\lambda_\theta}$ is in all the Schatten classes $S_p (\mathcal{D}_\alpha^{\, 2})$ for $p > 0$.
\end{theorem}

The proof shows that we can take $b = \sqrt{\alpha} \, b_\theta$, where $b_\theta$ is a positive constant depending only on $\theta$ and that the constant 
$b'$ can be taken equal to $c \, \frac{2 (1 - \theta)}{2 \alpha + (1 - \alpha) \theta} \, \alpha^{3/2}$ for some positive absolute constant $c$. 

\begin{proof}
Since $\lambda_\theta$ is univalent, its weighted Nevanlinna counting function is:
\begin{displaymath} 
\qquad \qquad \qquad N_{\lambda_\theta, \alpha} (w) = (1 - |\lambda_\theta^ {- 1} (w)| )^\alpha \qquad \text{for } w \in \Omega:= \lambda_\theta (\D)
\end{displaymath} 
and $0$ elsewhere. By \cite[Theorem~1]{Zorb}, $C_{\lambda_\theta}$ is compact on $\mathcal{D}_\alpha^{\, 2}$ if and only if:
\begin{displaymath} 
\sup_{\xi \in \T} \frac{1}{h^2} \int_{W (\xi, h)} \frac{N_{\lambda_\theta, \alpha} (w)}{(1 - |w|^2)^\alpha} \, dA (w) 
\converge_{h \to 0} 0 \, .
\end{displaymath} 
Since, for $w \in \Omega$:
\begin{equation} \label{comportement lens} 
1 - |\lambda_\theta^ {- 1} (w)| \approx (1 - |w|)^{1 / \theta} \, ,
\end{equation} 
we have:
\begin{displaymath} 
\int_{W (\xi, h)} \frac{N_{\lambda_\theta, \alpha} (w)}{(1 - |w|^2)^\alpha} \, dA (w) \approx h^{2 + \frac{1 - \theta}{\theta} \alpha} \, ;
\end{displaymath} 
so $C_{\lambda_\theta}$ is compact on ${\mathcal D}_\alpha^{\, 2}$ for $\alpha > 0$ , but is not compact on $\mathcal{D}_0^{\, 2}$.
\smallskip

For the estimates on approximation numbers, the proof follows the line of that of Theorem~\ref{diri}; hence we only sketch it. 
\smallskip\goodbreak

\noindent {\sl Lower estimate}

For  $0 < \alpha \leq 1$, we can use Theorem~\ref{theo compos op}, $3)$ $a)$ and \cite[Theorem~2.1]{LLQR-Israel}:
\begin{displaymath} 
\e ^{ - c \sqrt{n}} \lesssim \big[a_{2 n} \big(C_{\lambda_\theta}^{H^2} \big) \big]^2  
\lesssim a_n \big(C_{\lambda_\theta}^{{\mathcal D}_\alpha^{\, 2}} \big) \, ,
\end{displaymath} 
since $H^2 = {\mathcal D}_1^{\, 2}$ is dominated by ${\mathcal D}_\alpha^{\, 2}$. \smallskip

However, for all $\alpha > 0$, the proof given for the cusp map can be used also for the lens maps. The only difference is that, if  
$\lambda_\theta (u_j) = v_j$, we have $1 - u_j \approx (1 - v_j)^{1 / \theta}$, via \eqref{comportement lens}, and:
\begin{displaymath}
\qquad \qquad \qquad \qquad \lambda_\theta ' (z) \approx (1 - z)^{\theta  - 1} \qquad \text{for } z \in \D \text{ with } \Re z > 0 \, .
\end{displaymath}
Hence we get:
\begin{displaymath}
|\lambda_\theta ' (u_j)| \, \frac{\| K_{v_j}\|}{\| K_{u_j}\|} \gtrsim 
\frac{1}{(1 - v_j)^{\frac{1}{\theta} - 1}} \, 
\frac{(1 - v_j)^{\frac{1}{\theta} (1 + \frac{\alpha}{2})}}{(1 - v_j)^{1 + \frac{\alpha}{2}}} 
= (1 - v_j)^{\frac{\alpha}{2} (\frac{1}{\theta} - 1)} \, .
\end{displaymath}

Choosing $v_j = 1 - \e^{- j \eps}$, we get:
\begin{displaymath}
a_n^{{\mathfrak B}^2_\alpha} (M_{\lambda_\theta '} C_{\lambda_\theta}) 
\gtrsim \exp \bigg( - \bigg[ \frac{\alpha}{2} \Big(\frac{1}{\theta} - 1 \Big) \, n \, \eps + \frac{C}{\eps} \bigg] \bigg) \, .
\end{displaymath}
Taking $\eps = \sqrt{\frac{2 C}{\alpha} \frac{\theta}{1 - \theta}} \, \frac{1}{\sqrt{n}}$ gives now the result.
\medskip\goodbreak

\noindent {\sl Upper estimate.} 
\smallskip

$1)$ We have:
\begin{lemma}
For the lens map $\lambda_\theta$ of parameter $\theta$, $0 < \theta < 1$, we have, for $\alpha > 0$:
\begin{equation} 
\| \lambda_\theta^n \|_{\mathcal{D}_\alpha^{\, 2}} \lesssim n^{- \alpha / 2 \theta} \, . 
\end{equation} 
\end{lemma}
\begin{proof}
Let $\Omega = \{z \in \D \tq \Re z > 0\}$. For $z \in \Omega$, we easily see that:
\begin{equation} 
|\lambda_\theta (z)| \leq \exp ( - c \, | 1 - z|^\theta) \quad \text{and} \quad |\lambda'_\theta (z) | \lesssim |1 - z|^{\theta - 1} \, .
\end{equation} 
We get, using $(1 - |z|)^\alpha \leq |1 - z|^\alpha$ and the symmetry of $\lambda_\theta$:
\begin{align*}
\| \lambda_\theta^n \|_{\mathcal{D}_\alpha^{\, 2}}^2 
& = 2 (\alpha + 1) \int_\Omega n^2 |\lambda_\theta (z)|^{2 n - 2} |\lambda'_\theta (z)|^2 (1 - |z|^2)^\alpha \\
& \lesssim n^2 \int_\Omega \exp \big( - 2 \, c (n - 1) \, | 1 - z|^\theta \big) \, |1 - z|^{2 \theta - 2} \, |1 - z |^\alpha \, dA (z) \, .
\end{align*}
Using polar coordinates centered at $1$ and then making the change of variable $x = c n r^\theta$, so $r = c^{- 1 / \theta} n^{- 1 / \theta} x^{1 / \theta}$, 
we obtain:
\begin{align*}
\| \lambda_\theta^n \|_{\mathcal{D}_\alpha^{\, 2}}^2 
& \lesssim n^2 \int_0^{+ \infty} \exp (- c n r^\theta) \, r^{2 \theta - 2} \, r^\alpha \, r \, dr \\
& \approx n^2 \int_0^{+ \infty} 
\e^{ - x} n^{- \frac{2 \theta - 1 + \alpha}{\theta}} x^{\frac{2 \theta - 1 + \alpha}{\theta}} x^{\frac{1}{\theta} - 1} n^{- 1 / \theta} \, dx 
\approx n^{- \alpha / \theta} \, .
\end{align*}
\end{proof}
\smallskip

$2)$ We have:
\begin{equation} 
\| C_{\lambda_\theta} - C_{\lambda_\theta} S_N \| _{\mathcal{D}_\alpha^{\, 2}}^2 
\lesssim  \frac{1}{N^{\frac{1 - \theta}{\theta}\, \alpha}} \, \cdot
\end{equation} 

In fact, using the Hilbert-Schmidt norm on $\mathcal{D}_\alpha^{\, 2}$:
\begin{align*} 
\| C_{\lambda_\theta} - C_{\lambda_\theta} S_N \| _{\mathcal{D}_\alpha^{\, 2}}^2 
& \leq \| C_{\lambda_\theta} - C_{\lambda_\theta} S_N \| _{HS}^2 
\lesssim \sum_{n > N} \frac{1}{n^{\alpha / \theta}} \, \frac{1}{n^{1 - \alpha}} \\
& =  \sum_{n > N} \frac{1}{n^{\frac{1 - \theta}{\theta}\, \alpha  + 1}}  
\approx \frac{1}{N^{\frac{1 - \theta}{\theta} \, \alpha}}  \, \cdot
\end{align*} 
\smallskip

$3)$ We have, exactly as for the cusp map:
\begin{equation} 
a_n (C_{\lambda_\theta} S_N) \lesssim N^{\frac{\alpha + 1}{2}} \, a_n (J) \, ,
\end{equation} 
where $J \colon H^2 \to L^2 (\mu)$ is the canonical injection.
\smallskip

$4)$ We have:
\begin{equation} 
a_n (J) \lesssim \e^{ - c \, \sqrt{\alpha} \, \sqrt{n}} \, . 
\end{equation} 

In fact, we take the Blaschke product $B_0$ as for the cusp map, except that here we take its length $r$ as the largest integer $< \sqrt n$. We then take 
$B = B_0^{[\alpha \sqrt{n}]}$. With the notation used for the cusp map, when $l < r$ and $2^{- r} \leq |w - 1| < 2^{- l}$, we have 
$|B (w)| \lesssim \kappa^{\alpha \sqrt{n}}$; and when $l \geq r$, we use \eqref{comportement lens}. 
\smallskip

$5)$ Finally, we have:
\begin{displaymath} 
a_n (C_{\lambda_\theta}) \lesssim \frac{1}{N^{\frac{1 - \theta}{\theta} \, \alpha}} + N^{\frac{\alpha + 1}{2}} \, \e^{- c \, \sqrt{\alpha n}} \, ,
\end{displaymath} 
and the choice for $N$ of the integer part of $\e^{c \, \frac{2 \theta}{2 \alpha + (1 - \alpha) \theta} \, \sqrt{\alpha} \, \sqrt{n}}$ gives:
\begin{displaymath} 
a_n (C_{\lambda_\theta}) \lesssim \exp \bigg( - c \, \frac{2 (1 - \theta)}{2 \alpha + (1 - \alpha) \theta} \, \alpha^{3/2} \, \sqrt{n} \bigg) \, \cdot \qedhere
\end{displaymath} 
\end{proof}

\noindent {\bf Remark.} Since $a_n (C_{\lambda_\theta} ) \gtrsim \e^{- \alpha b_\theta \sqrt{n}}$ for $\alpha > 0$, 
Theorem~\ref{theo compos op}, $3)$ $a)$, gives:
\begin{displaymath}
a_{n}^{{\mathcal D}_0^{\, 2}} (C_{\lambda_\theta}) \gtrsim\Big( a_{2 n}^{{\mathcal D}_\alpha^{\, 2}} (C_{\lambda_\theta})\Big)^2 \gtrsim 
\e^{- 2\alpha b_\theta \sqrt{2 n}} \, ;
\end{displaymath}
letting $\alpha$ tend to $0$, we get $a_n^{{\mathcal D}_0^{\, 2}} (C_{\lambda_\theta}) \gtrsim 1$ and we recover that $C_{\lambda_\theta}$ 
is not compact on ${\mathcal D}_0^{\, 2}$.
\bigskip

\noindent{\bf Acknowledgement.} L. Rodr{\'\i}guez-Piazza is partially supported by the project PGC2018-094215-B-I00 
 (Spanish Ministerio de Ciencia, Innovaci\'on y Universidades, and FEDER funds). 
 
Parts of this paper was made when he visited the Universit\'e d'Artois in Lens and the Universit\'e de Lille in January 2019. It is his pleasure to 
thank all his colleagues in these universities for their warm welcome.

This work is also partially supported by the grant ANR-17-CE40-0021 of the French National Research Agency ANR (project Front).

\goodbreak

\smallskip
\vbox{
{\footnotesize
Pascal Lef\`evre \\
Univ. Artois, Laboratoire de Math\'ematiques de Lens (LML) UR~2462, \& F\'ed\'eration CNRS Nord-Pas-de-Calais FR~2956, 
Facult\'e Jean Perrin, Rue Jean Souvraz, S.P.\kern 1mm 18 
F-62\kern 1mm 300 LENS, FRANCE \\
pascal.lefevre@univ-artois.fr
\smallskip

Daniel Li \\ 
Univ. Artois, Laboratoire de Math\'ematiques de Lens (LML) UR~2462, \& F\'ed\'eration CNRS Nord-Pas-de-Calais FR~2956, 
Facult\'e Jean Perrin, Rue Jean Souvraz, S.P.\kern 1mm 18 
F-62\kern 1mm 300 LENS, FRANCE \\
daniel.li@univ-artois.fr
\smallskip

Herv\'e Queff\'elec \\
Univ. Lille Nord de France, USTL,  
Laboratoire Paul Painlev\'e U.M.R. CNRS 8524 \& F\'ed\'eration CNRS Nord-Pas-de-Calais FR~2956 
F-59\kern 1mm 655 VILLENEUVE D'ASCQ Cedex, FRANCE \\
Herve.Queffelec@univ-lille.fr
\smallskip
 
Luis Rodr{\'\i}guez-Piazza \\
Universidad de Sevilla, Facultad de Matem\'aticas, Departamento de An\'alisis Matem\'atico \& IMUS,  
Calle Tarfia s/n  
41\kern 1mm 012 SEVILLA, SPAIN \\
piazza@us.es
}
}


\begin{thebibliography} {99}

\bibitem {Attele} K.~R.~M.~Attele, 
Analytic multipliers of Bergman spaces, 
Michigan J. Math. 31 (1984), 307--319.

\bibitem {Attele-2} K.~R.~M.~Attele, 
Multipliers of the range of composition operators, 
Tokyo J. Math. 15, no. 1 (1992), 185--198.

\bibitem {Bak-Newman} J.~Bak, D.~J.~Newman, 
Complex analysis, Second edition, 
Undergraduate Texts in Mathematics, Springer-Verlag, New York (1997). 

\bibitem {Bayart-Queff-Seip} F.~Bayart, H.~Queff\'elec, K.~Seip, 
Approximation numbers of composition operators on $H^p$ spaces of Dirichlet series, 
Ann. Inst. Fourier (Grenoble) 66, no. 2 (2016), 551--588.

\bibitem {Carl-Stephani} B.~Carl, I.~Stephani, 
Entropy, compactness and the approximation of operators,
Cambridge Tracts in Mathematics 98, Cambridge University Press, Cambridge (1990). 

\bibitem {Isabelle} I.~Chalendar, J.~R.~Partington, 
Norm estimates for weighted composition operators on spaces of holomorphic functions, 
Complex Anal. Oper. Theory 8, no. 5 (2014), 1087--1095.

\bibitem {IsabelleII} I.~Chalendar, J.~R.~Partington, 
Compactness and norm estimates for weighted composition operators on spaces of holomorphic functions,  
Harmonic analysis, function theory, operator theory, and their applications, 81--89, Theta Ser. Adv. Math. 19, Theta, Bucharest (2017).

\bibitem {Contreras}  M.~D.~Contreras, A.~G.~Hern{\'a}ndez-D{\'\i}az,  
Weighted composition operators between different Hardy spaces, 
Integral Equations Operator Theory 46, no. 2 (2003), 165--188. 

\bibitem {Co-MacC} C.~C.~Cowen, B.~D.~MacCluer, 
Composition operators on spaces of analytic functions, 
Studies in Advanced Mathematics, CRC Press, Boca Raton, FL (1995).

\bibitem {DIE} J.~Dieudonn\'e, 
Calcul infinit\'esimal, deuxi\`eme \'edition, 
Hermann (1980). 

\bibitem {Primer-Dirichlet}  O.~El-Fallah, K.~Kellay, J.~Mashreghi, T.~Ransford,  
A primer on the Dirichlet space, 
Cambridge Tracts in Mathematics 203, Cambridge University Press, Cambridge (2014).

\bibitem {Garnett} J.~B.~Garnett, 
Bounded analytic functions, Revised first edition, Graduate Texts in Mathematics 236, Springer, New York (2007). 

\bibitem {Hastings} W.~W.~Hastings, 
A Carleson measure theorem for Bergman spaces, 
Proc. Amer. Math. Soc. 52 (1975), 237--241. 

\bibitem {Hayman} W.~K.~Hayman,  P.~B.~Kennedy,
Subharmonic functions, Vol. I, London Mathematical Society Monographs 9, Academic Press, London-New York (1976).

\bibitem {Kacnelson} V.~\`E~Kacnel'son, 
A remark on canonical factorization in certain spaces of analytic functions (Russian), 
in: Investigations on linear operators and the theory of functions III, edited by N.~K.~Nikol'ski{\u i}, 
Zap. Naucn. Sem. Leningrad. Otdel. Mat. Inst. Steklov (LOMI) 30 (1972), 163--164. 
Translation: J. Soviet Math. 4 (1975), no. 2 (1976), 444--445. 

\bibitem {Karim-Pascal} K.~Kellay, P.~Lef\`evre, 
Compact composition operators on weighted Hilbert spaces of analytic functions, 
J. Math. Anal. Appl. 386 (2012), 718--727.

\bibitem {Kerman-Sawyer}  R.~Kerman, E.~Sawyer,  
Carleson measures and multipliers of Dirichlet-type spaces, 
Trans. Amer. Math. Soc. 309, no.~1 (1988), 87--98. 

\bibitem {KRMA}T.~Kriete, B.~MacCluer, 
A rigidity theorem for composition operators on certain Bergman spaces,
Michigan. Math.~J. 42, no.~2 (1995), 379--386. 

\bibitem {GDHL}  G.~Lechner, D.~Li, H.~Queff\'elec, L.~Rodr{\'\i}guez-Piazza, 
Approximation numbers of weighted composition operators, 
J. Funct. Anal. 274, no.~7 (2018), 1928--1958.

\bibitem {LLQR-2008} P.~Lef\`evre, D.~Li, H.~Queff\'elec, L.~Rodr{\'\i}guez-Piazza,  
Some examples of compact composition operators on $H^2$, 
J. Funct. Anal. 255, no.~11 (2008), 3098--3124. 

\bibitem {LLQR-Math-Ann} P.~Lef\`evre, D.~Li, H.~Queff\'elec, L.~Rodr{\'\i}guez-Piazza,
Nevanlinna counting function and Carleson function of analytic maps, 
Math. Ann. 351, no.~2 (2011), 305--326. 

\bibitem {LLQR-TAMS} P.~Lef\`evre, D.~Li, H.~Queff\'elec, L.~Rodr{\'\i}guez-Piazza,  
Compact composition operators on Bergman-Orlicz spaces, 
Trans. Amer. Math. Soc. 365, no.~8 (2013), 3943--3970.

\bibitem {LLQR-Israel} P.~Lef\`evre, D.~Li, H.~Queff\'elec, L.~Rodr{\'\i}guez-Piazza,  
Some new properties of composition operators associated with lens maps, 
Israel J. Math. 195, no.~2 (2013), 801--824.

\bibitem {LLQR-revisited} P.~Lef\`evre, D.~Li, H.~Queff\'elec, L.~Rodr{\'\i}guez-Piazza,  
Some revisited results about composition operators on Hardy spaces, 
Revista Mat. Iberoamericana 28 (1) (2012), 57--76.


\bibitem {LLQR-JFA2013} P.~Lef\`evre, D.~Li, H.~Queff\'elec, L.~Rodr{\'\i}guez-Piazza,  
Compact composition operators on the Dirichlet space and capacity of sets of contact points, 
J. Funct. Anal. 264, no.~4 (2013), 895--919. 

\bibitem {LLQR-Arkiv} P.~Lef\`evre, D.~Li, H.~Queff\'elec, L.~Rodr{\'\i}guez-Piazza,  
Approximation numbers of composition operators on the Dirichlet space, 
Ark. Mat. 53, no.~1 (2015), 155--175.

\bibitem {LQR-JAT} D.~Li, H.~Queff\'elec, L.~Rodr{\'\i}guez-Piazza,
On approximation numbers of composition operators, 
J. Approx. Theory 164, no.~4 (2012), 431--459.

\bibitem {LQR-Fennicae} D.~Li, H.~Queff\'elec, L.~Rodr{\'\i}guez-Piazza,  
Estimates for approximation numbers of some classes of composition operators on the Hardy space, 
Ann. Acad. Sci. Fenn. Math. 38, no.~2 (2013), 547--564.

\bibitem {LQR-radius} D.~Li, H.~Queff\'elec, L.~Rodr{\'\i}guez-Piazza,  
A spectral radius type formula for approximation numbers of composition operators, 
J. Funct. Anal. 267, no.~12 (2014), 4753--4774.

\bibitem {Luecking}  D.~H.~Luecking, 
Trace ideal criteria for Toeplitz operators, 
J. Funct. Anal. 73, no.~2 (1987), 345--368. 

\bibitem {Luecking-Zhu} D.~H.~Luecking, K.~H.~Zhu, 
Composition operators belonging to the Schatten ideals, 
Amer. J. Math. 114, no. 5 (1992), 1127--1145. 

\bibitem {MacCluer} B.~D.~MacCluer, 
Compact composition operators on $H^p (B_N)$, 
Michigan Math. J. 32, no.~2 (1985), 237--248.

\bibitem {MacCluer-Shapiro} B.~D.~MacCluer, J.~H.~ Shapiro, 
Angular derivatives and compact composition operators on the Hardy and Bergman spaces,
Canad. J. Math. 38, no.~4 (1986), 878--906. 

\bibitem {Pau-Perez} J.~Pau, P.~A.~P\'erez,
Composition operators acting on weighted Dirichlet spaces,
J. Math. Anal. Appl. 401, no.~2 (2013), 682--694.

\bibitem {Shapiro-PAMS} J.~H.~Shapiro,
Compact composition operators on spaces of boundary-regular holomorphic functions, 
Proc. Amer. Math. Soc. 100, no.~1 (1987), 49--57.

\bibitem {Shapiro} J.~H.~Shapiro,
The essential norm of a composition operator, 
Ann. of Math. (2) 125, no.~2 (1987), 375--404.

\bibitem {Shapiro-livre} J.~H.~Shapiro, 
Composition operators and classical function theory, 
Universitext, Tracts in Mathematics, Springer-Verlag, New York (1993).

\bibitem {Shapiro-Taylor} J.~H.~Shapiro, P.~D.~Taylor, 
Compact, nuclear, and Hilbert-Schmidt composition operators on $H^2$, 
Indiana Univ. Math. J. 23 (1973/74), 471--496.

\bibitem {Simon} B.~Simon, 
Trace ideals and their applications, 
London Mathematical Society Lecture Note Series 35, Cambridge University Press, Cambridge-New York (1979).

\bibitem {Stegenga} D.~A.~Stegenga, 
Multipliers of the Dirichlet space, 
Illinois J. Math. 24, no.~1 (1980), 113--139.

\bibitem {Taylor} G.~D.~Taylor, 
Multipliers on $D_\alpha$, 
Trans. Amer. Math. Soc. 123 (1966), 229--240. 

\bibitem {Wu} Z.~Wu, 
Carleson measures and multipliers for Dirichlet spaces, 
J. Funct. Anal. 169, no. 1 (1999), 148--163.

\bibitem {Zhu} K.~Zhu,  
Operator theory in function spaces, Second edition, 
Mathematical Surveys and Monographs 138, American Mathematical Society, Providence, RI (2007).

\bibitem {Zorb} N.~Zorboska,
Composition operators on weighted Dirichlet spaces,
Proc. Amer. Math. Soc. 126, no.~7 (1998), 2013--2023.

\end{thebibliography}
\end{document}